\documentclass[12pt,oneside]{article}
\newcommand{\Path}{}
\usepackage{ \Path Paper_style}
\usepackage{ \Path Keywords_WB}
\usepackage{ \Path Category}
\usepackage{ \Path Environments}
\usepackage{tikz-cd}
\usepackage{mathtools}
\usepackage{\Path tikzcd2}
\usepackage{Paper_keywords}
\title{Functors induced by comma categories}

\begin{document}
\author{ Suddhasattwa Das \footnotemark[1] }
\footnotetext[1]{Department of Mathematics and Statistics, Texas Tech University, Texas, USA}
\date{\today}
\maketitle
\begin{abstract}
	Category theory provides a collective description of many arrangements in mathematics, such as topological spaces, Banach spaces and game theory. Within this collective description, the perspective from any individual member of the collection is provided by its associated left or right slice. The assignment of slices to objects extends to a functor from the base category, into the category of categories. Slice categories are a special case of the more general notion of comma categories. Comma categories are created when two categories $\mathcal{A}$ and $\mathcal{B}$ transform into a common third category $\calC$, via functors $F,G$. Such arrangements denoted as $\Comma{F}{G}$ abound in mathematics, and provide a categorical interpretation of many constructions in Mathematics. Objects in this category are morphisms between objects of $\mathcal{A}$ and $\mathcal{B}$, via the functors $F,G$. We show that these objects also have a natural interpretation as functors between slice categories of $\mathcal{A}$ and $\mathcal{B}$. Thus even though $\mathcal{A}$ and $\mathcal{B}$ may have completely disparate structures, some morphisms in $\calC$ lead to functors between their respective slices. We present this relation in the form of a functor from $\calC$ into the category of left slices. The proof of our main result requires a deeper look into associated categories, in which the objects themselves are various commuting diagrams.
\end{abstract}
Category theory has emerged as an useful alternative to the descriptive language of Set-theory. Instead of specifying mathematical objects from the ground-up, i.e. by their constituents, it provides a collective description of their relative arrangements. Categorical approaches have yielded surprising simplifications of deep results in all fields of mathematics, such as Topology \cite[e.g.]{Leinster2011selfsim, Das2024hmlgy, Rotman2013algtopo}, Probability theory \cite[e.g.]{Leinster_integration_2020, fritz2020Stoch, FritzRischel2020inf}, Dynamical systems theory \cite[e.g.]{Suda2022Poincare, Das2023CatEntropy, DasSuda2024recon}, and Game theory \cite[e.g.]{GhaniHEdges2006game, GhaniEtAl2018iter}. Readers can obtain a basic understanding of categories and functors from sources such as \cite[e.g.]{Maclane2013, Riehl_context_2017}.

Our discussion is based on the following general arrangement of categories and functors :
\begin{equation} \label{eqn:abc}
	\begin{tikzcd}
		\mathcal{A} \arrow[rd, "\alpha"'] & & \mathcal{B} \arrow[ld, "\beta"] \\
		& \calC & 
	\end{tikzcd}
\end{equation}
We shall see several examples of how such a general arrangement is prevalent all over Mathematics. Our focus will be on a category built upon such arrangements, called the \emph{comma category} $\Comma{\alpha}{\beta}$. The objects in this category are
\[ ob\paran{\Comma{\alpha}{\beta}} := \SetDef{ \paran{ a, b, \phi } }{ a\in ob(\calA), \, b\in ob(\calB), \, \phi \in \Hom_{\calC} \paran{ \alpha a ; \beta b } } , \]
and the morphisms comprise of pairs
$\SetDef{ (f,g) }{ f\in \Hom(D),\, g\in \Hom(E) }$ such that the following commutation holds :
\[\begin{tikzcd} \blue{ (a,\phi,b) } \arrow{r}{ (f,g) } & \akashi{ (a',\phi',b') } \end{tikzcd} \,\Leftrightarrow \,
\begin{tikzcd} a \arrow{d}{f} \\ a' \end{tikzcd}, \begin{tikzcd} b \arrow{d}{g} \\ b' \end{tikzcd}, \mbox{ s.t. }
\begin{tikzcd}
	\blue{ \alpha a } \arrow{r}{\alpha f} \arrow[blue]{d}{\phi} & \akashi{ \alpha a' } \arrow[Akashi]{d}{\phi'} \\
	\blue{ \beta b } \arrow{r}{\beta g} & \akashi{ \beta b' }
\end{tikzcd}\]
Thus $\Comma{\alpha}{\beta}$ may be interpreted as the category of bindings between $\alpha, \beta$, via their common codomain $\calC$. Comma categories contain as sub-structures, the original categories $\calA, \calB$, via the \emph{forgetful} functors
\[\begin{tikzcd} \calA & \Comma{\alpha}{\beta} \arrow{l}[swap]{\Forget_1} \arrow{r}{\Forget_2} & \calB \end{tikzcd}\]
whose action on morphisms in $\Comma{\alpha}{\beta}$ can be described as
\[ \begin{tikzcd} a \arrow{d}{f} \\ a' \end{tikzcd}
\begin{tikzcd} {} & & {} \arrow[ll, "\Forget_1"'] \end{tikzcd}
\begin{tikzcd}
	\blue{ \alpha a } \arrow{r}{\alpha f} \arrow[blue]{d}{\phi} & \akashi{ \alpha a' } \arrow[Akashi]{d}{\phi'} \\
	\blue{ \beta b } \arrow{r}{\beta g} & \akashi{ \beta b' }
\end{tikzcd}
\begin{tikzcd} {} \arrow[rr, "\Forget_2"] & & {} \end{tikzcd}
\begin{tikzcd} b \arrow{d}{g} \\ b' \end{tikzcd} \]
Comma categories prevail all over category theory and mathematics in general, such as graph theory \cite[e.g.]{goguen1991mnfst}, in the theory of lenses and fibrations \cite[e.g.]{johnson2012lenses}, iterative algebras \cite[e.g.]{adamek2006iterative}, stochastic processes \cite[e.g.]{cho2019disintegration}, Par{\'e} et. al.'s work on double categories \cite{grandis1999limits, grandis2004adjoint}, connectedness \cite[e.g.]{pare1973conn}, and mathematical logic \cite[e.g.]{streicher2004intro, pech2018fraisse}. If a category can be presented as a comma category, then one obtains additional results to prove the existence of (co)-limits \cite[e.g.]{GoguenBurstall1984a, Das2023CatEntropy}. We now examine two simpler examples for motivation :

\begin{example} [Measured dynamical systems] \label{ex:dyn_obs}
	Suppose $\calT$ is a semigroup, then it is representable as a 1-object category. Let $\Topo$ be the category of topological spaces and continuous maps. Then the class of topological dynamics is the functor category $\Functor{\calT}{\Topo}$. Now let $\EucCat$ be the subgroup of $\Topo$ comprised only of Euclidean spaces. Then the following arrangement is of the pattern given in \eqref{eqn:abc} :
	\[\begin{tikzcd}
		\Functor{\calT}{\Topo} \arrow[bend right = 20]{dr}[swap]{\dom} & & \EucCat \arrow[bend left = 20]{dl}{\subset} \\
		& \Topo
	\end{tikzcd}\]
	The domain functor $\dom$ above assigns to every dynamical system its domain. A typical object in the resulting comma category is a dynamical systems $(\Omega, \Phi^t)$ along with a measurement $\phi : \Omega\to \real^d$ :
	\[\begin{tikzcd} \Omega \arrow{r}{\Phi^t} & \Omega \arrow{r}{\phi} & \real^{d} \end{tikzcd}\]
	A typical morphism in this comma category is a change of variables $h:\Omega\to \Omega'$ that leads to the following joint commutation  :
	\[\begin{tikzcd} [row sep = large]
		\Omega \arrow{d}[swap]{h} \arrow[blue]{r}{\Phi^t} & \Omega \arrow{d}[swap]{h} \arrow[blue]{r}{\phi} & \real^{d} \arrow{d}[swap]{\iota A} \\
		\Omega' \arrow[Holud]{r}{\Phi'^t} & \Omega' \arrow[Holud]{r}{\phi'} & \real^{d'}
	\end{tikzcd} , \quad \forall t\in \calT .\]
	This comma category encapsulates the collection of all \emph{measured, topological dynamical systems}. This category has been instrumental in a categorical study of data and reconstruction theory \cite{DasSuda2024recon}.
\end{example}

\begin{example} [Finite subspaces] \label{ex:Affine:1}
	Let $\AffineCat$ be the category of vector spaces and affine maps, $\VectCat$ be the category of vectors spaces and linear maps, $\Euclid$ be the category of finite dimensional vector spaces and linear maps, and $\Euclid_{mono}$ be the subcategory of $\Euclid$ comprising only of injective maps. Then the following arrangement
	\[\begin{tikzcd}
		\Euclid_{mono} \arrow[bend right = 20]{dr}[swap]{\subset} & & \AffineCat \arrow[bend left = 20]{dl}{\proj} \\
		& \VectCat
	\end{tikzcd}\]
	leads to a comma category in which the objects are affine linear embeddings $P : \real^d \to V$ of finite dimensional vector spaces in possible infinite dimensional vector spaces. The morphisms between these objects are now allowed to be affine. Shown below is a morphism from an object $P$ to an object $P'$ :
	\[\begin{tikzcd}
		\real^d \arrow[d, "\iota"] \arrow[r, "P"] & V \arrow[d, "A"] \\
		\real^{d'} \arrow[r, "P'"] & V'
	\end{tikzcd}\]
	in which $\iota$ must be injective, and $A$ is an affine map. The study of these objects and morphisms lead to a generalized notion of "null" and "everywhere" \cite{HSYprevalence1992, Das2025null}.
\end{example}

A particular instance of comma categories are \emph{slice categories}. Henceforth, we shall use the symbol $\star$ to denote the category with a single object with no non-trivial morphism. Take any category $\calX$, and an object $x$ in it. This object may be interpreted by a unique functor from $\star$ to $\calX$, which we shall also denote by $\star \xrightarrow{x} \calX$. Now set 
\[ \calB = \star, \, \calA = \calC = \calX, \, \beta = x, \, \alpha = \Id_{\calX} , \]
in \eqref{eqn:abc}. The resulting comma category $\Comma{\Id_{\calX}}{x}$ is known as the left slice of $x$ in $\calX$, and will be denoted more briefly as $\Comma{\calX}{x}$. A typical morphism in this category is shown below
\[\begin{tikzcd}
	y \arrow[r, "f", blue] \arrow[d, "\phi"'] & x \\
	y' \arrow[ru, "f'"', Akashi] & 
\end{tikzcd}\]
The objects are the morphisms shown in blue, and a morphism $\phi$ from $f$ to $f'$ is a morphism $\phi:y\to y'$ such that the above commutation holds. One can similarly define the right slice of an object within its category. An important example of a slice category is the right slice of the pointed space in the category $\Topo$ of topological spaces. This corresponds to the category of pointed topological spaces. If $\calX$ is a preorder category, the left or right slice of an object $x$ is the \emph{down-set} or \emph{up-set} of the object. If $\calX$ is the collection of subsets of a superset $\calU$ ordered by inclusion, then the left slice of any subset $x$ of $\calU$ is the power set of $x$, also ordered by inclusion. 

\begin{example} \label{ex:powset:1}
	Let $\calU$ be any topological space. Then its left slice $\Comma{\Topo}{\calU}$ in $\Topo$ is the category formed by the collection of all continuous maps into $\calU$.
\end{example}

\begin{example} \label{ex:powset:2}
	Let $\calU$ be a set and $2^{\calU}$ be the power set of $\calU$. Thus $2^{\calU}$ is a preorder and a category. Then the left slice of any set $S\subset \calU$ in $2^{\calU}$ is the power set of $S$.
\end{example}

\begin{example} \label{ex:subob:1}
	More generally, let $\calC$ be any category and $\calC_{mono}$ be the subcategory comprising only of monomorphisms. Then the left slice of any object $c$ of $\calC$ in $\calC_{mono}$ is the \emph{subobject category} of $c$.
\end{example}

The forgetful functors inbuilt into comma categories may also be arranged into diagrams similar to \eqref{eqn:abc}. Consider the following more abstract example :

\begin{example} \label{ex:abstract:1}
	Given the arrangement as in \eqref{eqn:abc} one has the following diagram :
	\[\begin{tikzcd}
		\Comma{\alpha}{\beta} \arrow[rd, "\Forget_2"'] & & \star \arrow[ld, "b"] \\
		& \mathcal{B} & 
	\end{tikzcd}\]
	where $b$ is an object of $\calB$ and $\star$ is the 1-point category. This arrangement creates the comma category $\Comma{ \Forget_2^{ \Comma{\alpha}{\beta} } }{ b} $, whose objects are
	\begin{equation} \label{eqn:eodpc9}
		ob \paran{ \Comma{ \Forget_2^{ \Comma{\alpha}{\beta} } }{ b} }  := \braces{ \begin{tikzcd} \alpha a' \arrow[d, "f"] \\ \beta b' 	\end{tikzcd}, \; \begin{tikzcd} b' \arrow[d, "g"] \\ b \end{tikzcd} }
	\end{equation}
	Every such object creates an L-shaped diagram :
	\[\begin{tikzcd}
		\alpha a' \arrow[d, "f"'] \\ \beta b' \arrow[r, "\beta g"] & \beta b'
	\end{tikzcd}\]
	This is a diagram within $\calC$ in which one morphism $f$ is drawn from $\calC$, while another $g$ is drawn from $\calB$.
\end{example}

Example \ref{ex:abstract:1} reveals how various diagrams following a certain pattern create a category of its own. Example \ref{ex:abstract:1} is a generalization of the more concrete Example \ref{ex:dyn_obs}. The comma category \eqref{eqn:eodpc9} will play a significant role in the next section, where we state some technical results. 

Yet another important manifestation of comma categories are \emph{arrow} categories. If we set
\[ \calA = \calB = \calC = \calX, \, \alpha = \beta = \Id_{\calX} , \]
in \eqref{eqn:abc}, then the resulting comma category $\Comma{\Id_{\calX}}{\Id_{\calX}}$ is called the arrow category of $\calX$, and is denoted by $\ArrowCat{\calX}$. The objects in this category are the arrows or morphisms in $\calX$. A morphism between two morphisms $x \xrightarrow{f} x'$ and $y \xrightarrow{g} y'$ is a pair of morphisms $x \xrightarrow{\phi} y$ and $x' \xrightarrow{\phi'} y'$ such that the following commutation holds :
\[\begin{tikzcd}
	x \arrow[r, "f"] \arrow[d, "\phi"'] & x' \arrow[d, "\phi'"] \\
	y \arrow[r, "g"] & y' 
\end{tikzcd}\]
Thus $\ArrowCat{\calX}$ reveals how the arrows of $\calX$ are bound to each other via the commutation relations in $\calX$. Some important examples are

\begin{example} \label{ex:Topo:1}
	The arrow category $\ArrowCat{\Topo}$ in $\Topo$ corresponds to the category of topological pairs. 
\end{example}

\begin{example} \label{ex:Vect:1}
	The arrow category $\ArrowCat{\VectCat}$ in $\VectCat$ corresponds to the category of linear maps, with linear change of variables serving as morphisms. 
\end{example}

\begin{example} \label{ex:powset:5}
	Recall the notations in Examples \ref{ex:powset:1} and \ref{ex:powset:2}. Let $\calU$ be a topological space. Take $\calX$ to be the power set $2^{\calU}$ and $\calY$ to be the subcategory of $\Topo$ formed by all topological subspaces of $\calU$. Then one has an obvious inclusion $\iota : \calX \to \calY$. Note that the comma category $\Comma{\iota}{\iota}$ has the same objects as the arrow category $\ArrowCat{\calY}$ : any object $F\in \ArrowCat{\calY}$ is a continuous map between topological spaces $F:X\to Y$. But a morphism from $F$ to another such object $F' : X' \to Y'$ is just inclusion, i.e., $X'\supseteq X$, $Y'\supseteq Y$ and $F$ is a restriction of $F'$.
\end{example}

Comma, slice and arrow categories thus represent finer structures present within categories, and also how objects from different categories assemble together to produce more complex categories. The language of comma categories contributes to the universality of category theory as an alternate formulation for descriptive set theory \cite{Lawvere1963semantics}. 
\blue{ Comma categories have been used in a variety of ways, from being a descriptive tool to higher constructions in Category theory. Classical expositions on category theory \cite[e.g.]{Maclane2013,Riehl_context_2017} rely on slice categories for pointwise description of Kan extensions. The descriptive strength has also been an integral part of a categorical reformulation of Dynamical systems theory \cite[e.g.]{DasSuda2024recon, DasSuda2025enrich}. One important consideration in our analysis is the fact that the objects of comma and slice categories are morphisms, and thus composable in nature. This composability of comma-objects were utilized in \cite{nakahira2023diagrammatic} to develop  a string-diagrammatic language for ordinary categories. One of the most important applications of comma categories is Par{\'e} et. al.'s work on \emph{double categories} \cite{grandis1999limits, grandis2004adjoint}, which are categories with two orthogonal dimensions of structure. This very concept, along with its associated notion of \emph{connectedness} \cite[e.g.]{pare1973conn} for categories, rely on a heavy use of comma categories. In the next section, we present the main result which analyzes how the objects in a comma category could induce functors between left slices. The functor is created when one tries to complete "L"-shaped diagrams into universal commuting squares. The topic of functors induced by universal properties of comma categories has not been explored much. The reader might find interest in the recent work \cite{HuZhu2022special} on functors induced by comma categories involving \emph{exact functors} and \emph{Abelian categories}.}

\section{Main results} \label{sec:result}

The main difference of a categorical description of a subject from the classical set-theoretic description, is that properties are not internal to an object, but defined entirely in terms of their relations to external objects. Set-theoretic discourses start with several elementary ideas such as points, sets, maps, numbers and addition, and more advanced ideas are built from various combinations of these primitive concepts. On the other hand category theory only has morphisms and objects as the primitive concepts. So the more advanced concepts of mathematics have to be realized as patterns or diagrams in a categorical discourse. Various conclusions and theorems are extracted from the inter-relations between various diagram classes, and from properties defined by \emph{universality}.

The simplest diagram in a category is a morphism, which expresses a relation between two objects. A more advanced diagram would be a commutation square, 
\[\begin{tikzcd} A \arrow[d, "f"] \arrow[r, "g"] & C \arrow[d, "h"] \\ B \arrow[r, "i"] & D \end{tikzcd} \imply 
\begin{tikzcd} A \arrow[d, "f"]  & {} \\ B \arrow[r, "i"] & D \end{tikzcd} , \; 
\begin{tikzcd} A \arrow[r, "g"] & C \arrow[d, "h"] \\ {} & D \end{tikzcd} , \; 
\begin{tikzcd} {} & C \arrow[d, "h"] \\ B \arrow[r, "i"] & D \end{tikzcd} ,\;
\begin{tikzcd} A \arrow[d, "f"] \arrow[r, "g"] & C  \\ B  & {} \end{tikzcd} \]
which not only expresses a set of four relations between four objects, but also a relation binding these four relations. Every such square contains four smaller L shaped diagrams as shown above. They are the four corners of the commutation square and are objects of certain comma categories. As objects of these categories they are also inter-related functorially. Adjacent corners overlap on a common arrow. These overlaps can be expressed as constraints that bind the functors between the various comma-categories. The paper presents a deep dive into this diagrammatic language of category theory. One can various questions, such as given any corner L-diagram, what is the minimal or maximum square to which it can be completed ? If the bottom edge $i$ is fixed, one can look for lower-left corners, lower right corners that extend that edge into an $L$-diagram. Let these classes be named $LL(i)$ and $LR(i)$ respectively. Are these classes also categories ? If each object in $LL(i)$ and $LR(i)$ can be completed into a universal square, is this correspondence functorial, and how does it correspond to the $i$ we started out with ? While these questions are entirely diagrammatic, we will discover important applications to classical branches of Mathematics.

We focus on an arrangement of the form $\calX  \xrightarrow{\iota} \calY \xleftarrow{\iota} \calX$, which is a special instance of \eqref{eqn:abc}. More precisely

\begin{Assumption} \label{A:1}
	There are complete categories $\calX$ and $\calY$, with initial objects $0_{\calX}$ and $0_{\calY}$ respectively, and there is a continuous functor $\iota : \calX \to \calY$ such that $0_{\calY} = \iota \paran{ 0_{\calX} }$, and $\iota$ is injective on objects.
\end{Assumption}

Let $\LSlice{ \calX }$ denote the category whose objects are left-slice categories $ \Comma{\calX}{\Omega}$ for various $X\in\calX$, and morphisms are the functors between these categories. Thus $\LSlice{ \calX }$ is a full subcategory of 
$\CatCat$, the category of small categories. Let $\iota(\calX)$ denote the full subcategory of $\calY$ generated by objects of $\iota$. Note that the morphisms in $\iota(\calX)$ are precisely the objects of $\Comma{\iota}{\iota}$. Recall that a morphism $f$ in any category is said to be \emph{surjective} or equivalently, an \emph{epimorphism}, if for any composable morphisms $g,g'$, if $g'f = gf$, then $g=g'$. Similarly, a morphism is said to be \emph{injective} or equivalently, a \emph{monomorphism}, if for any composable morphisms $g,g'$, if $fg' = fg$, then $g=g'$. We need the following assumptions : 

\begin{Assumption} \label{A:3}
	For every monomorphism $f$ in $\calY$, there are morphisms $g$ in $\calY$ and $h$ in $\calY$ such that $f = (\iota h) g$.
\end{Assumption}

\begin{Assumption} \label{A:4}
	The image under $\iota$ of every morphism in $\calX$ is injective in $\calY$.
\end{Assumption}

\begin{Assumption} \label{A:5}
	The category $\calY$ is balanced, i.e., any morphism in $\calY$ which is both surjective and injective is an isomorphism.
\end{Assumption}

Assumption \ref{A:5} is satisfied in categories such as topoi \cite{Goldblatt2014topoi, Freyd1972topoi, Brook1975fin}. Two prime examples of topoi are $\SetCat$ and $\Topo$. Another important category in which Assumption \ref{A:5} is satisfied is $\GroupCat$, the category of groups and homomorphisms. Our main result establishes a functor 
\begin{equation} \label{eqn:thm:0}
	\iota(\calX) \xrightarrow{\Dyn} \LSlice{ \calX }
\end{equation}
that achieves certain universal diagram completions, as discussed before. Recall that a typical morphism in $\iota(\calX)$ is a $\calY$-morphism $\iota \Omega \xrightarrow{F} \iota \Omega'$. The functor in \eqref{eqn:thm:0} should convert this into a functor between the left-slice categories $\Comma{\calX}{\Omega}$ and $\Comma{\calX}{\Omega'}$ associated to the endpoints of $F$.

\begin{theorem} \label{thm:0}
	Let Assumptions \ref{A:1}, \ref{A:3}, \ref{A:4} and \ref{A:5} hold. Then there is a functor as in \eqref{eqn:thm:0}	which maps every object $\Omega\in \calX$ to the left slice $\Comma{\calX}{\Omega}$. It maps an morphism $\iota \Omega \xrightarrow{F} \iota \Omega'$ into a functor $$\Comma{\calX}{\Omega} \xrightarrow{\tau_F} \Comma{\calX}{\Omega'}$$ such that for every slice-object $A \xrightarrow{a} \Omega \in \Comma{\calX}{\Omega}$, $\tau_F(a)$ creates a commutation square
	\begin{equation} \label{eqn:thm:0:1}
		\begin{tikzcd}
			\iota A \arrow[d, "\iota a"']  \\
			\iota \Omega \arrow[r, "F"'] & \iota \Omega'               
		\end{tikzcd} \imply 
		\begin{tikzcd}
			\iota A \arrow[d, "\iota a"'] \arrow[r, "\tau_F(a)", Shobuj] & \iota B \arrow[d, "\iota b", Shobuj] \\
			\iota \Omega \arrow[r, "F"'] & \iota \Omega'               
		\end{tikzcd}
	\end{equation}
	for some  $B \xrightarrow{b} \Omega \in \Comma{\calX}{\Omega'}$. Moreover, this square is universal in the sense that for any slice-object $b':B'\to \Omega'$ and any morphism $f:\iota A \to \iota B'$, if the blue commutation square shown below holds 
	\begin{equation} \label{eqn:thm:0:2}
		\forall \begin{tikzcd} 	\end{tikzcd}
		\begin{tikzcd}
			& & \iota B' \arrow[ldd, "\iota b'", bend left=49, Akashi] \\
			\iota A \arrow[d, "\iota a"', Akashi] \arrow[r, "\tau_F(a)", Shobuj] \arrow[rru, "f", bend left, Akashi] & \iota B \arrow[d, "\iota b", Shobuj] \arrow[ur, "\exists!\phi"', Holud] & \\
			\iota \Omega \arrow[r, "F"', Akashi] & \iota \Omega' & 
		\end{tikzcd} ,
	\end{equation}
	then there is a unique morphism $\phi$ which factors the outer commutation loop into the inner commutation loop.
\end{theorem}

The statement of Theorem \ref{thm:0} is pictorial. It says that beginning with any L-shaped diagram as shown on the left of \eqref{eqn:thm:0:1}, one has the commuting diagram as shown on the right of \eqref{eqn:thm:0:1}. Moreover, this diagram is universal in the sense that all other possible commutation squares over the L-diagram can be recovered from it, as shown in \eqref{eqn:thm:0:2}. The completion into a square is essentially created by the top horizontal arrow $\tau_F(a)$, which is being claimed by Theorem \ref{thm:0} to be the image of a functor. The left slice objects in $\LSlice{ \calX }$ reside within the structure of $\calX$. The claim thus implies that morphisms from a different category $\calY$ naturally creates functors between the various left-slices of $\calX$. 

There are numerous examples of the arrangement of Theorem \ref{thm:0} in Mathematics. One of the most important among them is the following : 

\begin{example} \label{ex:powset:3}
	Recall the notations in Example \ref{ex:powset:5}. Any object $F\in \Comma{\iota}{\iota}$ is a continuous map between topological spaces $F:X\to Y$. Then $\tau_F$ can be interpreted to be the induced map from the power set $\Comma{\iota}{X}$ of $X$, into the power set $\Comma{\iota}{Y}$ of $Y$. The correspondence between $F$ and $\tau_F$ is itself functorial. Overall, we have a functor $\Dyn : \Comma{\iota}{\iota} \to \LSlice{ 2^{\calU} }$. Note that $\LSlice{ 2^{\calU} }$ is the category in which each object is a power set of some $S\subset \calU$, and morphisms are inclusion preserving maps between these power sets.
\end{example}

Example \ref{ex:powset:3} can be generalized based on the following observation : 

\begin{lemma} \label{lem:di0l}
	Given any complete category $\calC$, the subcategory $\calC_{mono}$ formed by monomorphisms is complete, and has the same initial object as $\calC$.
\end{lemma}

Lemma \ref{lem:di0l} and Theorem \ref{thm:0} have the following important consequence : 

\begin{corollary} \label{corr:5}
	Suppose $\calC$ is a complete, balanced category. Then :
	\begin{enumerate} [(i)]
		\item The inclusion $\iota : \calC_{mono} \to \calC$ satisfies Assumptions \ref{A:1}, \ref{A:3}, \ref{A:4} and \ref{A:5}.
		\item There is a functor $\Dyn$ from $\calC$ into the full subcategory of $\CatCat$ spanned by the sub-object categories of $\calC$.
	\end{enumerate}
\end{corollary}

Some examples of categories $\calC$ which satisfy the conditions of Corollary \ref{corr:5} are $\SetCat$, $\Topo$ and $\GroupCat$. As a result we have the following instances of Corollary \ref{corr:5} :

\begin{example} \label{ex:Vect:2}
	Continuing the discussion in Example \ref{ex:Vect:1}, any linear map $A: U\to V$ between vector spaces induces a mapping between the subspaces of $U$ and $V$ respectively. This correspondence is a functor from $\VectCat$ to the category of collections of vector subspaces.
\end{example}

As a slight variation to Example \ref{ex:Vect:2} we have :

\begin{example} \label{ex:Vect:3}
	Take $\calX = \VectCat_{mono}$ and $\calY = \AffineCat$. The initial objects $0_\calX$ and $0_{\calY}$ are both the zero-dimensional vector space. Note that there is an inclusion functor $\iota : \calX \to \calY$ which maps $0_\calX$ into $0_{\calY}$. Thus according to Theorem \ref{thm:0} any affine map is functorially related to a map between the collection of vector subspaces of the corresponding spaces.
\end{example}

If $\calX$ is a preorder, a left slice $\Comma{\calX}{a}$ can be interpreted as the set $\SetDef{a' \in \calX}{ a'\leq a }$. This is a sub-preorder of $\calX$, and is called the \emph{down-set} of $a$.

\begin{corollary} \label{corr:6}
	Let $\calX$ be a complete preorder, and $\iota : \calX \to \GroupCat$ a functor that maps $0_{\calX}$ into the trivial group with one object. Let $a,b$ be objects in $\calX$ and $F: \iota(a) \to \iota(b)$ be a group homomorphism. Then $F$ induces an order preserving map $\tau_F$ between the down-sets $\Comma{\calX}{a}$ and $\Comma{\calX}{b}$. Moreover this correspondence is functorial.
\end{corollary}

An example of $\calX, \iota$ from Corollary \ref{corr:6} is when $\calX$ is a cellular complex, comprised of a collection of inclusions between face maps. Recall that singular homology is a functor
\[ \text{Homology} : \Topo \to \GroupCat , \]
mapping each topological space into its sequence of homology groups.

\begin{example} \label{ex:topo:1}
	Let $\calX$ be a simplicial complex of dimension $n$, comprised of simplices of various dimensions from $0$ to $n$. Each simplex of dimension less than $n$ is included as a face map of a simplex of a higher dimension. Thus $\calX$ is a preorder with a finite number of objects and morphisms contained within $\Topo$. Then for any two faces $a, b$ of $\calX$, the down-sets $\Comma{\calX}{a}$ and $\Comma{\calX}{b}$ are the sub-simplexes of these faces. Any group homomorphism
	\[ F : \text{Homology}(a) \to \text{Homology}(b) \]
	induces a unique simplicial map between these sub-complexes whose induced map between the homology groups is precisely $F$.
\end{example}

The last statement in the example above is supported by the commutation in \eqref{eqn:thm:0:2}. 

Example \ref{ex:dyn_obs} presented a functorial interpretation of dynamical systems. Category theoretic reformulations of dynamical systems have become of increasing interest due to the simplicity of presentation of many of the deeper results in dynamical systems theory \cite[e.g.]{MossPerrone2022ergdc, Suda2022Poincare, Das2023CatEntropy, DasSuda2024recon}. This is the advantage provided by the constructive / synthetic language of category theory, as opposed to the descriptive nature of set-theoretic language. The new challenge that emerges is that many basic definitions which are trivial in a set-theoretic presentation, becomes harder to present in a category theoretic setting. A prime example is the notion of an orbit. One can associate orbits to both topological, smooth, or measurable dynamical systems. However, the notion of orbit itself is as a minimal object in $\SetCat$ satisfying certain properties. In a category theoretic presentation, objects lose all their inner details and are presented as part of a larger collection. The focus shifts from the content of objects, to their relational and compositional structure. Thus orbits cannot be simply defined to be a union of successive images. A categorical definition of orbits is a major gap in the category theoretic reformulation of dynamics, and the functor $\Dyn$ discovered in Theorem \ref{thm:0} fills this gap.

\begin{example} \label{ex:powset:4}
	Recall the notations from Examples \ref{ex:powset:1}, \ref{ex:powset:2} and \ref{ex:powset:3}. Let $\text{Topo}(\calU)$ denote the subcategory of $\Topo$ generated by all topological subspaces of $\calU$. Recall from Example \ref{ex:dyn_obs} that a functor $\Phi: \calT \to \text{Topo}(\calU)$ is a topological dynamical system in the universe $\calU$. Its time semigroup $\calT$ is typically $\num_0$, $\integer$ or $\real$. Then one has the following composable sequence of functors
	\[\begin{tikzcd}
		\calT \arrow[rr, "\Phi"] && \text{Topo}(\calU) \arrow[rr, "\Dyn"] && \LSlice{2^{\calU}} 
	\end{tikzcd}\]
	The composition of these functors is a set-theoretic dynamical system with time semigroup $\calT$. Thus the interpretation of a topological dynamical system as a set-theoretic dynamical system is also functorial.
\end{example}

Dynamical systems theory is about the study of orbits, and their asymptotic properties. There has been a recent interest in developing a categorical language for dynamical systems, and a major gap has been a functorial description of orbits. Example \ref{ex:powset:4} along with Theorem \ref{thm:0} fills this gap.
Given a dynamical system $\Phi^t : \Omega \to \Omega$ on a space $\Omega$, the orbit of a subset $S\subseteq\Omega$ is the union $\cup_{t\in \calT} \Phi^t(S)$. Alternatively it can be defined to be the smallest subset of $\Omega$ that contains all the images $\Phi^t(S)$. This definition is very suitable for a category theoretic presentation. Note that the composite functor $\tilde\Phi : \Dyn \circ \Phi$ from Example \ref{ex:powset:4} leads to a functor $2^{\Omega} \times \calT \to 2^{\Omega}$. Then the orbit of $\Phi$ is the functor shown in the diagram below :
\begin{equation} \label{eqn:def:orbit}
	\begin{tikzcd}
		2^{\Omega} \times \calT \arrow[d, "\proj_1"'] \arrow[dashed]{drr}[name = n2]{} \arrow{rr}[name = n1]{\tilde\Phi} && 2^{\Omega} \\
		2^{\Omega} \arrow[rr, "\text{Orbit}"] && 2^{\Omega}
		\arrow[shorten <=1pt, shorten >=1pt, Rightarrow, to path={(n2) to[out=90,in=270] (n1)} ]{  }
	\end{tikzcd}
\end{equation}
The bottom horizontal arrow is created by a construction called a right-\emph{Kan extension}. Kan extensions are a purely diagrammatic / categorical notion, and is elaborated in Section \ref{sec:LimPre}. Example \ref{ex:powset:4} and Diagram \eqref{eqn:def:orbit} is a succinct but precise way of stating the following facts : 

\begin{corollary} \label{corr:7}
	Every topological dynamical system in the universe $\calU$  is a functor $\Phi : \calT \to \text{Topo}(\calU)$. It leads to the following notions :
	\begin{enumerate} [(i)]
		\item This functor combines with $\Dyn$ from Theorem \ref{thm:0} to get a dynamical system $\calT \to \LSlice{ 2^{\calU} }$.
		\item This leads to a functor $\tilde\Phi : 2^{\Omega} \times \calT \to 2^{\Omega}$.
		\item The existence of the orbit functor follows from the existence of a right Kan extension, as shown in \eqref{eqn:def:orbit}.
		\item The minimality of the orbit follows from the universal property of a right Kan extension.
	\end{enumerate}
\end{corollary}

These examples and Corollaries \ref{corr:5}--\ref{corr:7} highlight the prevalence of the arrangement described in Theorem \ref{thm:0}. This ends the presentation of some examples of manifestations of slice categories and applications of Theorem \ref{thm:0}. Slice categories have had recent applications \cite{dos2001opti} in optimization and approximation theory \cite{article_1247239, aral2023weighted, DasGiannakis2023harmonic}. The possibility of applications of Theorem \ref{thm:0} to this emerging field is an interesting prospect.

\paragraph{Outline} \blue{Theorem \ref{thm:0} has several layers to it. Firstly it associates a functor between left slice categories to every morphism of the form $F:\iota \Omega \to \iota \Omega'$. Secondly \eqref{eqn:thm:0} states that this correspondence itself is functorial, which means that composition of morphisms become composition of functors. Thirdly, the correspondence is defined by the unique and minimal commutation square \eqref{eqn:thm:0:1} that it creates. We shall unravel the categorical principles that contribute to each claim, over the course of the next sections. The main ingredient of Theorem \ref{thm:0} is Assumption \ref{A:1}. Assumption \ref{A:1} is slightly generalized into Assumption \ref{A:2} next in following Section \ref{sec:induced}. This generalization  allows us to formulate two Theorems \ref{thm:1} and \ref{thm:2} which cover part of the claims of Theorem \ref{thm:0}. The compositionality is proved next in Section \ref{sec:algebra} via Theorem \ref{thm:4}. Theorem \ref{thm:0} is proved in Section \ref{sec:algebra}, as a consequence of Theorems \ref{thm:1}, \ref{thm:2} and \ref{thm:4}. See Figure \ref{fig:outline1} for an outline of the proof of Theorem \ref{thm:0}, and how the other main results fit into the proof. We take a deep look at comma and arrow categories in Sections \ref{sec:comma_arr} and \ref{sec:comma_adj}. The uniqueness of the commutation in \eqref{eqn:thm:0:1} is next established via a special categorical construction in Section \ref{sec:LimPre}. Finally Theorems \ref{thm:1} and \ref{thm:2} are proved in Section \ref{sec:proof:2}.}

\section{The induced functor between slices} \label{sec:induced}

Theorem \ref{thm:0} was about the comma category $\Comma{\iota}{\iota}$ which is a special case of \eqref{eqn:abc}. We now make an assumption on \eqref{eqn:abc}, which turns out to be a generalization of Assumption \ref{A:1}. 

\begin{Assumption} \label{A:2}
	The category $\calA$ and $\calB$ from \eqref{eqn:abc} are complete, the functor $\beta$ is continuous. Categories $\calA$ and $\calC$ have initial elements $0_{\calA}$ and $0_{\calC}$ respectively, and $\alpha \paran{0_{\calA}} = 0_{\calC}$
\end{Assumption}	

Our first result arises from the simple situation when two objects $a, b$  are picked from $\calA, \calB$ in \eqref{eqn:abc}, mapped into $\calC$, and bound by a morphism $F$ in $\calC$. The objects $a, b$ have their own left-slice categories in $\calA, \calB$, which are independent of each other as well as $\calC$. We shall see how the morphism $\phi$ induces a functor between these two categories.

\begin{theorem} [Induced functor] \label{thm:1}
	Assume the arrangement of \eqref{eqn:abc}, and let Assumption \ref{A:2} hold. Fix an object $\alpha a \xrightarrow{F} \beta b$ of the comma category $\Comma{\alpha}{\beta}$. Then there is a functor $\tau_{F} : \Comma{\calA}{a} \to \Comma{\alpha}{\beta}$ such that for any object $a'\xrightarrow{f} a$ in $\Comma{\calA}{a}$, there is an object $b'\xrightarrow{g} b$ in $\Comma{\calB}{b}$, such that the following commutation holds 
	\begin{equation} \label{eqn:thm:1a}
		\begin{tikzcd} a' \arrow[d, "f"'] \\ b \end{tikzcd}
		\imply
		\begin{tikzcd}
			\alpha a' \arrow[d, "\alpha f"'] \arrow[r, "\tau_{F}(f)"] & \beta b' \arrow[d, "\beta g"] \\
			\alpha a \arrow[r, "F"'] & \beta b 
		\end{tikzcd}
	\end{equation}
	Moreover, $\tau_{F}(f)$ is minimal in the sense for any other object $b'' \xrightarrow{g} b$, if the commutation shown below on the left holds :
	\[
	\begin{tikzcd}
		\alpha a' \arrow[d, "\alpha f"'] \arrow[r, "\tilde{F}"] & \beta b'' \arrow[d, "\beta g''"] \\
		\alpha a \arrow[r, "F"'] & \beta b 
	\end{tikzcd} \imply 
	\begin{tikzcd}
		& & \beta b' \arrow[ldd, "\beta g'", bend left=50] \arrow[ld, "\beta \phi", dotted] \\
		\alpha a' \arrow[d, "\alpha f"'] \arrow[r, "\tilde{F}"'] \arrow[rru, "\tau_{F}(f)", bend left] & \beta b'' \arrow[d, "\beta g''"] & \\
		\alpha a \arrow[r, "F"'] & \beta b & 
	\end{tikzcd}
	\]
	then there is a unique morphism $b' \xrightarrow{\phi} b''$ such that the commutation on the right holds.
\end{theorem}

Thus the correspondence $\tau_{F}$ associates to every object $f$ in the left slice $\Comma{\calA}{a}$ an object $\tau_{F}(f)$ in the comma category $\Comma{\alpha}{\beta}$. This object $\tau_{F}(f)$ itself is an morphism in $\calC$ and creates a commutation square involving $f$ and $F$.

\paragraph{Remark} The minimality so described is hardly surprising, since whenever a commutation such as \eqref{eqn:thm:1a} holds, the following commutation also holds
\[\begin{tikzcd}
	& & & \beta b' \arrow[ldd, "\beta g", bend left=50] \arrow[ld, "\beta g", dotted] \\
	\alpha a' \arrow[d, "\alpha f"'] \arrow[rr, "\beta g \circ \tau_{F}(f)"'] \arrow[rrru, "\tau_{F}(f)", bend left] & & \beta b \arrow[d, "\beta \Id_{b}"] & \\
	\alpha a \arrow[rr, "F"'] & & \beta b & 
\end{tikzcd}\]
This diagram is a special case of the second claim of Theorem \ref{thm:1}, with $\tilde F = \beta g \circ \tau_{F}(f)$ and $\phi=g'$.

\paragraph{Remark} One of the consequences of Theorem \ref{thm:1} and the commutation in \eqref{eqn:thm:1a} is 
\[\begin{tikzcd} 
	& & \Comma{\calA}{a} \arrow{dll}[swap]{\Forget_1} \arrow{d}{\tau_F} \\ 
	\calA & & \Comma{\alpha}{\beta} \arrow{ll}{\Forget_1}    
\end{tikzcd}\]
This means that the domain of the morphism $\tau_{F}(f)$ is the same as the domain of the morphism $f$.

\paragraph{Remark} Any object $a' \xrightarrow{f} a$ in $\Comma{\calA}{a}$ is sent by $\Forget_1$ into $a'$, whereas it is sent by $\tau_F$ into $\tau_F(f)$, which is then sent by $\Forget_1$ into $a'$. This commutation be extended as follows : 
\begin{equation} \label{eqn:def:Dyn}
	\begin{tikzcd} 
		& & \Comma{\calA}{a} \arrow{rr}{\Dyn_F} \arrow{dll}[swap]{\Forget_1} \arrow{d}{\tau_F} & & \Comma{\calB}{b} \arrow{d}{ \Forget_1 } \\ 
		\calA & & \Comma{\alpha}{\beta} \arrow{ll}{\Forget_1} \arrow{rr}[swap]{\Forget_2} & & \calB
	\end{tikzcd}
\end{equation}
The  diagram presents a new functor $\Dyn_F$ between the slice categories associated to the terminal points of the comma object $F$.  

Recall the category $\Comma{ \Forget_2^{ \Comma{\alpha}{\beta} } }{ b }$ \eqref{eqn:eodpc9} presented in Example \ref{ex:abstract:1}. The compound objects in $\Comma{ \Forget_2^{ \Comma{\alpha}{\beta} } }{ b }$ lead to a projection functor
\[\begin{tikzcd} \Comma{ \Forget_2^{ \Comma{\alpha}{\beta} } }{ b } \arrow[rrr, "\text{Restrict}"'] &&& \Comma{\calB}{b} \end{tikzcd}\]
Both Theorem \ref{thm:1} and \eqref{eqn:def:Dyn} are consequences of the following more general result :

\begin{theorem} \label{thm:2}
	Under the same assumptions as Theorem~\ref{thm:1}  and the category in  \eqref{eqn:eodpc9} there is a functor
	\[ \bar{\tau}_F : \Comma{\calA}{a} \to \Comma{ \Forget_2^{ \Comma{\alpha}{\beta} } }{ b } \]
	such that the functors $\tau_{F}$ and $\Dyn_{F}$ are created via composition :
	\begin{equation}  \label{eqn:thm:2}
		\begin{tikzcd}
			&& \Comma{\calA}{a} \arrow[dll, "\tau_F"', dashed] \arrow[rr, "\Dyn_F", dashed] \arrow[drr, "\bar{\tau}_F"] && \Comma{\calB}{b} \\
			\Comma{\alpha}{\beta} && \UR{\alpha}{\beta}{\calB} \arrow[ll, "\Forget_1"] && \Comma{ \Forget_2^{ \Comma{\alpha}{\beta} } }{b} \arrow[ll, "\subseteq"] \arrow[u, "\text{Restrict}"']
		\end{tikzcd}
	\end{equation}
\end{theorem}

Theorem \ref{thm:2} jointly implies the statements of Theorem \ref{thm:1} and \eqref{eqn:def:Dyn}. Theorem \ref{thm:2} is proved in Section \ref{sec:proof:2}. See Figure \ref{fig:outline1} for a summary of the various results and their logical connections.

\paragraph{Remark} When $\calA = \calB = \calC$ in \eqref{eqn:abc}, and $\alpha = \beta = \Id_{\calA}$, then $\Comma{\alpha}{\beta}$ is just the arrow category $\ArrowCat{\calA}$. Any object $a \xrightarrow{F} b$ in this category induces a functor between the slice categories :
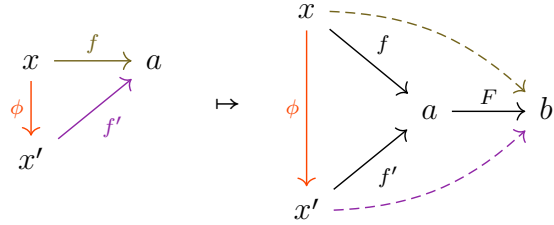
\[\begin{tikzcd}
	x \arrow[r, "f", Holud] \arrow[d, "\phi"', Itranga] & a \\
	x' \arrow[ru, "f'"', Akashi] & 
\end{tikzcd} \quad
\begin{tikzcd} {} \arrow[rrr, mapsto] &&& {} \end{tikzcd} \quad 
\begin{tikzcd}
	x \arrow[rrd, Holud, dashed, bend left = 20] \arrow[rd, "f"] \arrow[dd, "\phi"', Itranga]  \\
	& a \arrow{r}{F} & b \\
	x' \arrow[rru, Akashi, dashed, bend right = 20] \arrow[ru, "f'"'] & 
\end{tikzcd}\]
The yellow and blue arrows represent different objects in the respective slice categories, and the red arrows represent morphisms between these objects. The diagram on the right is obtained from the left by simply composing with $F$. This functorial relation coincides with $\tau_F$.

\paragraph{Remark} While $\tau_F$ has a simple interpretation when all the functors in \eqref{eqn:abc} are identities, determining an induced functor in the more general setting is not trivial. One notable approach relies on the existence of special factorization systems \cite[e.g.]{adamek1990abstract, FreydKelly1972cont1}. This approach has been extended to an axiomatic study of topology \cite{Schlomiuk1970topo, ClementinoGiuliTholen1996, DikranjanGiuli1987closure, DikranjanTholen2013clos}.

In the next section, we look more closely at the correspondence between $F$ and $\tau_F$.
\section{Algebra of induced functors} \label{sec:algebra}

Theorem \ref{thm:1} presents how an object in a comma category induces a functor between the left- slices associated to the two endpoints of the object. The functor is realized through morphisms in $\calC$ binding an object in a left slice object in $\calA$, to a left slice object in $\calB$. In this section we shift our attention back to the case when $\calA=\calB$. In that case all the  left slices involved are within the same category. Our first important realization will be that the induced morphisms $\tau_{F}(f)$ are surjective.

\begin{theorem} \label{thm:3}
	Assumptions \ref{A:1} and \ref{A:3} hold. Then the induced morphisms $\tau_F(f)$ from Theorem \ref{thm:1} are surjective.
\end{theorem} 

The proof requires the following lemma :

\begin{lemma} \label{lem:Eq_inj}
	In any category, an equalizer is injective.
\end{lemma}  

\paragraph{Proof of Theorem \ref{thm:3}} To prove surjectivity we need to show that for any pair of morphisms $\alpha, \beta : \iota B \to C$, if $\alpha \circ \tau_F{f} = \beta \circ \tau_F{f}$ then $\alpha = \beta$. Since $\calY$ is complete it has equalizers. Consider the following diagram in which the equalizer of $\alpha, \beta$ has been shown. 
\[\begin{tikzcd}
	& D \arrow{d}{\text{Eq}(\alpha, \beta)"} \arrow[rrd, bend left] & & \\
	\iota A \arrow[r, "\tau_F(a)"'] \arrow[ru, "\exists! \phi", dotted] & \iota B \arrow[r, "\alpha"'] \arrow[rd, "\beta"', bend right] & C & C \arrow[l, "\cong"] \arrow[ld, "\cong", bend left] \\
	& & C & 
\end{tikzcd}\]
Since the equalizer is by definition, the universal morphism $\gamma$ such that $\beta \gamma = \alpha \gamma$, the morphism $\tau_F(f)$ must factor through the equalizer via the morphism $\phi$ as shown. Now by Lemma \ref{lem:Eq_inj}, the morphism $\text{Eq}(\alpha, \beta)$ is injective. By Assumption \ref{A:3}, $\text{Eq}(\alpha, \beta)$ factorizes as shown below.
\[\begin{tikzcd}
	& & D \arrow{d}{\text{Eq}(\alpha, \beta)"} \arrow[rr, "f"] & & \iota E \arrow[lld, "\iota \psi"] \\
	\iota A \arrow[rr, "\tau_F(a)"'] \arrow[rru, "\phi"] & & \iota B & & 
\end{tikzcd}\]
This commutation diagram can be joined with the definition of $\tau_F(f)$  to get 
\[\begin{tikzcd}
	& & D \arrow{d}{\text{Eq}(\alpha, \beta)"} \arrow[rr, "f"] & & \iota E \arrow[lld, "\iota \psi"] \arrow[lldd, "\iota (b\circ \psi)", dotted, bend left] \\
	\iota A \arrow[rr, "\tau_F(a)"'] \arrow[rru, "\phi"] \arrow[d, "\iota a"'] \arrow[rrrru, dashed, bend left=50] & & \iota B \arrow[d, "\iota b"] & & \\
	\iota \Omega \arrow[rr, "F"'] & & \iota \Omega' & & 
\end{tikzcd}\]
By the universality of $\tau_F(f)$, the morphism $\psi$ must be an isomorphism. This would mean that $\text{Eq}(\alpha, \beta)$ is an isomorphism too. This in turn implies that $\alpha=\beta$, which was our goal. This completes the proof of Theorem \ref{thm:3}. \qed 

We have been examining the particular instance of \eqref{eqn:abc} when $\calA = \calB = \calX$, and $\calC = \calY$, and both functors $\alpha, \beta$ are  $\iota : \calX \to \calY$. In that case, the diagram \eqref{eqn:def:Dyn} becomes
\[\begin{tikzcd}
	\calX & \Comma{\iota}{\iota} \arrow[l, "\Forget_1"'] \arrow[r, "\Forget_2"] & \calX \\
	\Comma{\calX}{\Omega} \arrow[u, "\Forget_2"] \arrow[ru, "\tau_F"'] \arrow[rr, "\Dyn_F", dotted] & & \Comma{\calX}{\Omega'} \arrow[u, "\Forget_2"]
\end{tikzcd}\]
One of the consequences of equating $\calA$ and $\calB$ is that the functor described by Theorems \ref{thm:1} and \ref{thm:2} are between slices of the same category. Our goal is to investigate the composability of the horizontal arrows in the bottom row. To gain a precise footing, we assume

\begin{figure} [!t]
	\begin{tikzpicture}[scale=0.55, transform shape, framed, background rectangle/.style={double, ultra thick, draw=gray, rounded corners}]
		\node (thm0) [style={rect6}] at (0, -\rowA) {Theorem \ref{thm:0} : Functor from codomain category into category of left slices };
		\node (thm1) [style={rect6}] at (1.2\columnA, \rowA) { Theorem \ref{thm:1} : induced functor between left slice categories };
		\node (thm2) [style={rect6}] at (1.2\columnA, 0) { Theorem \ref{thm:2} : Commutation \eqref{eqn:thm:2} };
		\node (thm3) [style={rect6}] at (1.2\columnA, -\rowA) { Theorem \ref{thm:3} : surjectivity of the induced morphism };
		\node (thm4) [style={rect6}] at (1.2\columnA, -2\rowA) { Theorem \ref{thm:4} : compositionality };
		\node (A1) [style={rect5}] at (0.0\columnA, 0.0\rowA) { Assumption \ref{A:1} };
		\node (A2) [style={rect5}] at (0.0\columnA, 1.0\rowA) { Assumption \ref{A:2} };
		\node (A3) [style={rect5}] at (2.4\columnA, -1.0\rowA) { Assumption \ref{A:3} };
		\node (A4) [style={rect5}] at (2.4\columnA, -2\rowA) { Assumption \ref{A:4} };
		\node (A5) [style={rect5}] at (0, -2\rowA) { Assumption \ref{A:5} };
		\node (master) [style={rect2}] at (2.4\columnA, 0) { Construction \eqref{eqn:nc0s3} };
		\node (6) [style={rect2}] at (2.4\columnA, \rowA) { Lim-Pre construction \eqref{eqn:LimPre} };
		\node (7) [style={rect2}] at (3.5\columnA, 1\rowA) { Corner categories \eqref{eqn:ArrowComma3} };
		\node (8) [style={rect2}] at (3.5\columnA, 0.0\rowA) { Lemmas \ref{lem:frgt_adjnt}, \ref{lem:jd9lk3} };
		\node (9) [style={rect2}] at (3.5\columnA, -1\rowA) { Lemma \ref{lem:LimPre:extend} };
		\node (10) [style={rect2}] at (3.5\columnA, -2\rowA) { Pullback \eqref{eqn:T_T2_T3} };
		\draw[-to] (A2) to (thm1);
		\draw[-to] (A2) to (A1);
		\draw[-to] (A1) to (thm0);
		\draw[-to] (A3) to (thm3);
		\draw[-to] (thm1) to (thm0);
		\draw[-to] (thm2) to (thm0);
		\draw[-to] (thm3) to (thm4);
		\draw[-to] (thm4) to (thm0);
		\draw[-to] (A4) to (thm4);
		\draw[-to] (A5) to (thm4);
		\draw[-to] (master) to (thm2);
		\draw[-to] (6) to (thm1);
		\draw[-to] (7) to (master);
		\draw[-to] (8) to (master);
		\draw[-to] (9) to (master);
		\draw[-to] (10) to (master);
	\end{tikzpicture}
	\caption{Outline of the results, assumptions, and their logical dependence. The white boxes display the various assumptions, and grey boxes display the main results.}
	\label{fig:outline1}
\end{figure}
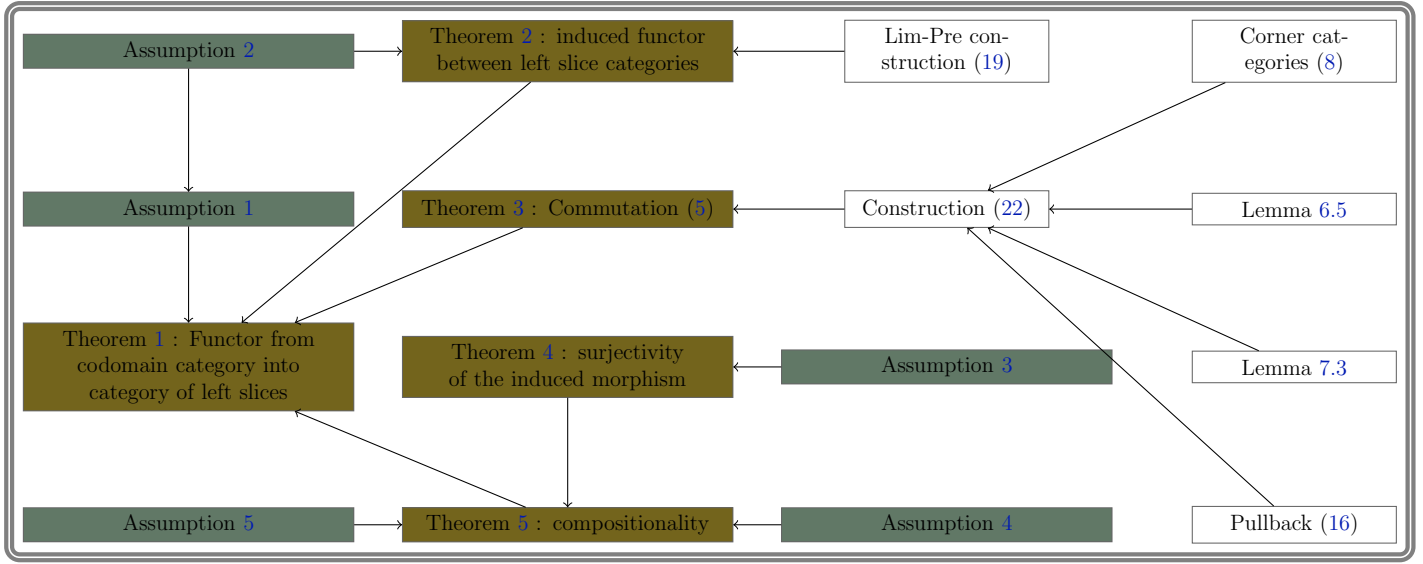

\begin{theorem}[Compositionality of induced functors] \label{thm:4}
	Suppose Assumptions \ref{A:1} \ref{A:4} and \ref{A:5} hold. Then there is a functor 
	\[ \iota(\calX) \xrightarrow{\tau} \LSlice{ \calX } , \]
	which maps an morphism $\iota \Omega \xrightarrow{F} \iota \Omega'$ into $\Comma{\calX}{\Omega} \xrightarrow{\tau_F} \Comma{\calX}{\Omega'}$.
\end{theorem}

\paragraph{Remark} Theorem \ref{thm:4} essentially says that the correspondence of $\Dyn_F$ with $F$ preserves composition. This leads to the following diagram :
\[ \begin{tikzcd} \iota \Omega \arrow[dashed,  bend right = 40]{dd}[swap]{F'\circ F} \arrow{d}{F} \\ \iota \Omega' \arrow{d}{F'} \\ \iota \Omega'' \end{tikzcd} \imply 
\begin{tikzcd}
	\Comma{\calX}{\Omega} \arrow[d, "\tau_F"] \arrow[rr, "\Dyn_F"] \arrow[dd, "\tau_{F'\circ F}"', bend right=49] \arrow[rrrr, "\Dyn_{F'\circ F}", bend left=30] & & \Comma{\calX}{\Omega'} \arrow[d, "\tau_F"] \arrow[ld, "\Forget_1"'] \arrow[rr, "\Dyn_{F'}"] & & \Comma{\calX}{\Omega''} \arrow[ld, "\Forget_21"'] \arrow[dd, "\Forget_1"] \\
	\Comma{\iota}{\iota} \arrow[r, "\Forget_2"] & \calX & \Comma{\iota}{\iota} \arrow[r, "\Forget_2"] & \calX & \\
	\Comma{\iota}{\iota} \arrow[rrrr, "\Forget_2"] & & & & \calX 
\end{tikzcd}\]
The upper commuting loop is the statement of Theorem \ref{thm:4}. The outer commutating loop, along with the two smaller loops are a consequence of \eqref{eqn:def:Dyn}.

\begin{lemma} \label{lem:surj_right_cancel}
	In any category $\calC$, if $f,g$ are two composable morphisms such that $f, g\circ f$ are surjective, then $g$ is also surjective.
\end{lemma} 

\begin{proof} We need to shown that for any morphisms $\alpha, \beta$ such that $\alpha g = \beta g$, $\alpha$ must equal $\beta$. Now note that
	\[ \alpha \paran{g f} = \paran{\alpha g} f = \paran{\beta g} f = \beta \paran{gf} . \]
	Since $gf$ is surjective, we must have $\alpha=\beta$, proving the claim. 
\end{proof}

\paragraph{Proof of Theorem \ref{thm:4}} We start with the following setup : %
\[\begin{tikzcd}
	\iota A \arrow[d, "\iota a"']  \\
	\iota \Omega \arrow[rr, "F"'] & & \iota \Omega' \arrow[rr, "F'"'] & & \iota \Omega'' & 
\end{tikzcd}\]
This contains an object $A\in \Comma{\calX}{\Omega}$, and two composable morphisms $F, F' \in \Comma{\iota}{\iota}$. We can apply the functors $\tau_F$ and $\tau_{F'}$ in succession to get
\[\begin{tikzcd}
	\iota A \arrow[d, "\iota a"'] \arrow[rr, "\tau_{F} (a)"] & & \iota B \arrow[d, "\iota b"] \arrow[rr, "\tau_{F'} (b)"] & & \iota C \arrow[d, "\iota c"']  & \\ 
	\iota \Omega \arrow[rr, "F"'] & & \iota \Omega' \arrow[rr, "F'"'] & & \iota \Omega'' & 
\end{tikzcd}\]
To prove Theorem \ref{thm:4}, it has to be shown that the composition along the morphisms in the upper row equals $\tau_{F'\circ F}$. The object $\tau_{F'\circ F}(a)$ itself can be drawn as shown below :
\[\begin{tikzcd}
	& & & & & \iota D \arrow[ldd, "\iota d", bend left]  \arrow[ld, "! \iota \phi"'] \\
	\iota A \arrow[d, "\iota a"'] \arrow[rr, "\tau_{F} (a)"] \arrow[rrrrru, "\tau_{F'\circ F} (a)", bend left] & & \iota B \arrow[d, "\iota b"] \arrow[rr, "\tau_{F'} (b)"] & & \iota C \arrow[d, "\iota c"'] & \\
	\iota \Omega \arrow[rr, "F"'] & & \iota \Omega' \arrow[rr, "F'"'] & & \iota \Omega'' & 
\end{tikzcd}\]
The connecting morphism $\phi : D \to C$ exists by the minimality of $\tau_{F'\circ F}(a)$. The upper commuting loop can be expressed as
\[ \tau_{F'} (b) \circ \tau_{F} (a) = \iota \phi \circ \tau_{F'\circ F} (a) . \]
By Theorem \ref{thm:3}, all the three morphisms $\tau_{F} (a)$, $\tau_{F'} (b)$ and $\tau_{F'\circ F} (a)$ are surjective. Thus by Lemma \ref{lem:surj_right_cancel}, $\iota \phi$ must be surjective too. By Assumption \ref{A:4}, $\iota \phi$ is also injective. Thus by Assumption \ref{A:5}, $\iota \phi$ is an isomorphism. This implies that $\tau_{F'} (b) \circ \tau_{F} (a)$ and $\tau_{F'\circ F} (a)$ are equal up to isomorphism. This completes the proof of Theorem \ref{thm:4}. \qed 

This completes the statement of our main results. The proof of Theorem \ref{thm:0} can now be completed.

\paragraph{Proof of Theorem \ref{thm:0}} Note that Assumption \ref{A:1} in Theorem \ref{thm:0} is a special case of Assumption \ref{A:2}. As a result we can build the induced functors $\bar{\tau}_F$ and $\tau_F$ from Theorems \ref{thm:1} and \ref{thm:2} respectively. Since $\iota$ is assumed to be injective on objects, the objects of $\iota(\calX)$ are in bijection with the object of $\calX$. Thus each object in $\iota(\calX)$ corresponds to a unique image $\iota \Omega$, for some $\Omega \in ob(\calX)$. The functoriality now follows from Theorem \ref{thm:4}. This completes the proof of Theorem \ref{thm:0}. \qed 

Theorems \ref{thm:1} and \ref{thm:2} remain to be proven. The proofs require building a deeper insight into the inter-relations between comma, arrow, and slice categories. We build this insight over the course of three sections \ref{sec:comma_arr}, \ref{sec:comma_adj} and \ref{sec:LimPre}. In the next section, we complete the proof of Theorem \ref{thm:0}.

\section{Comma and arrow categories} \label{sec:comma_arr}

In this section we take a deeper look into the commutation squares in comma categories. We assume throughout this section the general arrangement of \eqref{eqn:abc}, and the resultant comma category $\Comma{\alpha}{\beta}$. We have seen how an  arrow category is a special instance of a comma category. In this section we are interested in the arrow category of the comma category : $\ArrowCat{ \Comma{\alpha}{\beta} }$. The objects of this category are commutations of the form
\begin{equation} \label{eqn:ijdo4}
	\begin{tikzcd}
		\alpha a \arrow[r, "\phi"] \arrow[d, "\alpha f"'] & \beta b \arrow[d, "\beta g"] \\
		\alpha a' \arrow[r, "\phi'"] & \beta b' 
	\end{tikzcd}, \quad a,a' \in ob(\calA) , \; b, b' \in ob(\calB) .
\end{equation}
The vertical morphisms lie in $\calC$ while the horizontal morphisms are the images of morphisms in $\calA$ and $\calB$. The key to proving our results is the realization that the different pieces of \eqref{eqn:ijdo4} are also comma categories of various kinds. Let us consider the lower left and upper right corners of \eqref{eqn:ijdo4} :
\[\begin{tikzcd} \alpha a \arrow[d, "\alpha f"'] \\ \alpha a' \arrow[r, "\phi'"] & \beta b' \end{tikzcd} , \quad \quad 
\begin{tikzcd} \alpha a \arrow[r, "\phi"] & \beta b \arrow[d, "\beta g"] \\ & \beta b' \end{tikzcd}\]
This first diagram is an object of the comma category 
\[ \DL{\alpha}{\beta}{\calA} := \Comma{ \Id_{\calA} }{ \Forget_1^{ \Comma{\alpha}{\beta} } } \]
The initials DL indicates "down-left", the position of an object of this category relative to an object of $\ArrowCat{ \Comma{\alpha}{\beta} }$ \eqref{eqn:ijdo4}. Similarly, the upper-right corner is an object of the category
\[ \UR{\alpha}{\beta}{\calB} := \Comma{ \Forget_2^{ \Comma{\alpha}{\beta} } }{ \Id_{\calB} } . \]
Both the categories $\DL{\alpha}{\beta}{\calA}$ and $\UR{\alpha}{\beta}{\calB}$ can be written more expressively as
\[\left[ \begin{tikzcd} \calA \arrow[bend right=20]{dr}[swap]{\Id} & & \Comma{\alpha}{\beta} \arrow[bend left=20]{dl}{\Forget_1} \\ & \calA \end{tikzcd}\right] , \quad \quad \left[ \begin{tikzcd} \Comma{\alpha}{\beta} \arrow[bend right=20]{dr}[swap]{\Forget_2} & & \calB \arrow[bend left=20]{dl}{\Id} \\ & \calB \end{tikzcd}\right] . \]
One can proceed similarly to describe each of the other two corners of \eqref{eqn:ijdo4} as categories. This leads to the following layout of the arrow category and its corner categories :
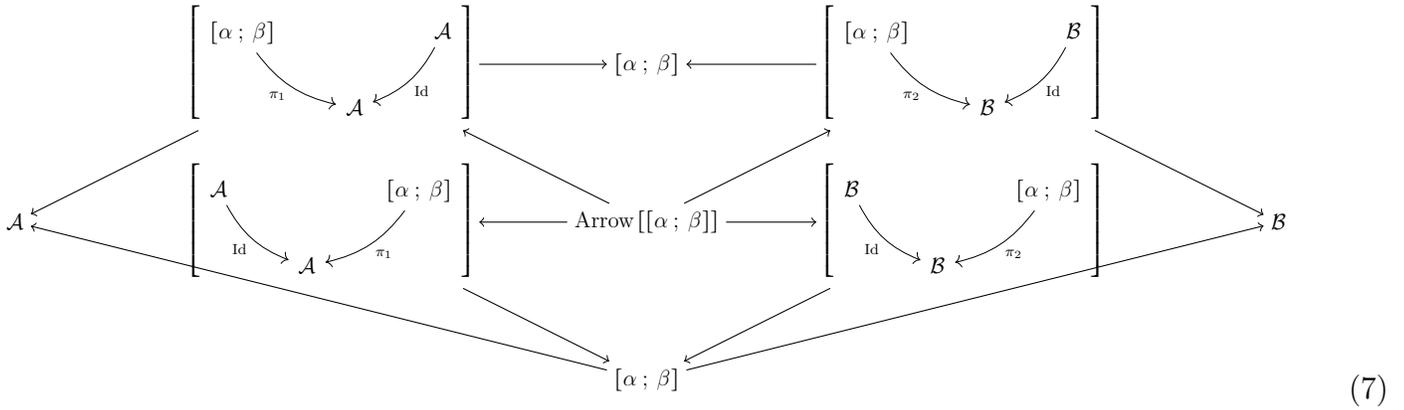
\begin{equation} \label{eqn:ArrowComma2}
	\begin{tikzpicture}[scale=0.5, transform shape]
		\node (0) at (\columnA, -\rowA) {$\ArrowCat{ \Comma{\alpha}{\beta} }$};
		\node (dl) at (0, -\rowA) {$\left[ \begin{tikzcd} \calA \arrow[bend right=20]{dr}[swap]{\Id} & & \Comma{\alpha}{\beta} \arrow[bend left=20]{dl}{\Forget_1} \\ & \calA \end{tikzcd}\right]$};
		\node (ul) at (0, 0.5\rowA) {$\left[ \begin{tikzcd} \Comma{\alpha}{\beta} \arrow[bend right=20]{dr}[swap]{\Forget_1} & & \calA \arrow[bend left=20]{dl}{\Id} \\ & \calA \end{tikzcd}\right]$};
		\node (ur) at (2\columnA, 0.5\rowA) {$\left[ \begin{tikzcd} \Comma{\alpha}{\beta} \arrow[bend right=20]{dr}[swap]{\Forget_2} & & \calB \arrow[bend left=20]{dl}{\Id} \\ & \calB \end{tikzcd}\right]$};
		\node (dr) at (2\columnA, -\rowA) {$\left[ \begin{tikzcd} \calB \arrow[bend right=20]{dr}[swap]{\Id} & & \Comma{\alpha}{\beta} \arrow[bend left=20]{dl}{\Forget_2} \\ & \calB \end{tikzcd}\right]$};
		\node (lm) at (-\columnA, -1.5\rowA) {$\calA$};
		\node (rm) at (3\columnA, -1.5\rowA) {$\calB$};
		\node (dm) at (\columnA, -3\rowA) {$\Comma{\alpha}{\beta}$};
		\node (um) at (\columnA, 0.5\rowA) {$\Comma{\alpha}{\beta}$};
		\draw[-to] (0) to (dl);
		\draw[-to] (0) to (ul);
		\draw[-to] (0) to (dr);
		\draw[-to] (0) to (ur);
		\draw[-to] (ul) to (um);
		\draw[-to] (ur) to (um);
		\draw[-to] (dl) to (dm);
		\draw[-to] (dr) to (dm);
		\draw[-to] (ul) to (lm);
		\draw[-to] (ur) to (rm);
		\draw[-to] (dm) to (lm);
		\draw[-to] (dm) to (rm);
	\end{tikzpicture}
\end{equation}
The arrows connecting the comma categories are functors, created from the forgetfull functors associated with the arrow category. The commutative diagram in \eqref{eqn:ijdo4} is an object in the central category of this diagram. The image of \eqref{eqn:ijdo4} under the various functors of \eqref{eqn:ArrowComma2} are displayed below :
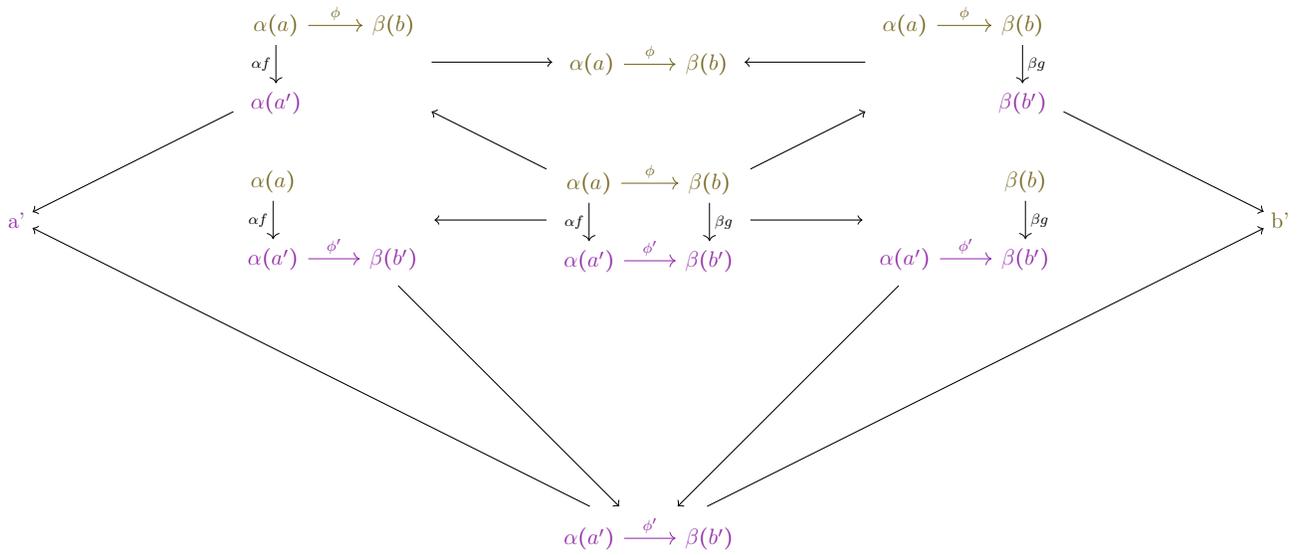
\[\begin{tikzpicture}[scale=0.5, transform shape]
	\node (0) at (\columnA, -\rowA) { \begin{tikzcd} \Holud{\alpha(a)} \arrow{d}[swap]{\alpha f} \arrow[Holud]{r}{\phi} & \Holud{\beta(b)} \arrow{d}{\beta g} \\ \akashi{ \alpha(a') } \arrow[Akashi]{r}{\phi'} & \akashi{ \beta(b') } \end{tikzcd} };
	\node (dl) at (0, -\rowA) {\begin{tikzcd} \Holud{\alpha(a)} \arrow{d}[swap]{\alpha f} \\ \akashi{ \alpha(a') } \arrow[Akashi]{r}{\phi'} & \akashi{ \beta(b') } \end{tikzcd}};
	\node (ul) at (0, 0.5\rowA) {\begin{tikzcd} \Holud{\alpha(a)} \arrow{d}[swap]{\alpha f} \arrow[Holud]{r}{\phi} & \Holud{\beta(b)} \\ \akashi{ \alpha(a') } \end{tikzcd}};
	\node (ur) at (2\columnA, 0.5\rowA) {\begin{tikzcd} \Holud{ \alpha(a) } \arrow[Holud]{r}{\phi} & \Holud{ \beta(b) } \arrow{d}{\beta g} \\ & \akashi{ \beta(b') } \end{tikzcd}};
	\node (dr) at (2\columnA, -\rowA) {\begin{tikzcd} & \Holud{\beta(b)} \arrow{d}{\beta g} \\ \akashi{ \alpha(a') } \arrow[Akashi]{r}{\phi'} & \akashi{ \beta(b') } \end{tikzcd}};
	\node (lm) at (-\columnA, -1.5\rowA) {\akashi{a'}};
	\node (rm) at (3\columnA, -1.5\rowA) {\Holud{b'}};
	\node (dm) at (\columnA, -3\rowA) {\begin{tikzcd} \akashi{ \alpha(a') } \arrow[Akashi]{r}{\phi'} & \akashi{ \beta(b') } \end{tikzcd}};
	\node (um) at (\columnA, 0.5\rowA) {\begin{tikzcd} \Holud{\alpha(a)} \arrow[Holud]{r}{\phi} & \Holud{\beta(b)} \end{tikzcd}};
	\draw[-to] (0) to (dl);
	\draw[-to] (0) to (ul);
	\draw[-to] (0) to (dr);
	\draw[-to] (0) to (ur);
	\draw[-to] (ul) to (um);
	\draw[-to] (ur) to (um);
	\draw[-to] (dl) to (dm);
	\draw[-to] (dr) to (dm);
	\draw[-to] (ul) to (lm);
	\draw[-to] (ur) to (rm);
	\draw[-to] (dm) to (lm);
	\draw[-to] (dm) to (rm);
\end{tikzpicture}\]
The corner categories, which have been presented pictorially, can be written more succinctly as comma categories :
\begin{equation} \label{eqn:ArrowComma3}
	\begin{tikzcd}[column sep = large, row sep = large]
		\calA & \Comma{\alpha}{\beta} \arrow{l}[swap]{\Forget^{ \Comma{\alpha}{\beta} }_1} \arrow{r}{\Forget^{ \Comma{\alpha}{\beta} }_2} & \calB \\
		\Comma{ \Forget^{ \Comma{\alpha}{\beta} }_1 }{ \Id_{\calA} } \arrow{ur}[swap]{\Forget_1} \arrow[Akashi, bend right = 90]{ddd}[swap]{\Forget_2} &  & \Comma{ \Forget^{ \Comma{\alpha}{\beta} }_2 }{ \Id_{\beta} } \arrow{ul}{\Forget_1} \arrow[Akashi, bend left = 90]{ddd}{\Forget_2} \\
		& \ArrowCat{ \Comma{\alpha}{\beta} } \arrow[Shobuj]{uu}{U_{\alpha, \beta}} \arrow[Shobuj]{dd}{D_{\alpha, \beta}} \arrow[Shobuj]{dr}{DR_{\alpha, \beta}} \arrow[Shobuj]{ur}[swap]{UR_{\alpha, \beta}} \arrow[Shobuj]{dl}[swap]{DL_{\alpha, \beta}} \arrow[Shobuj]{ul}{UL_{\alpha, \beta}} & \\ 
		\Comma{ \Id_{\calA} }{ \Forget^{ \Comma{\alpha}{\beta} }_1 } \arrow{dr}{\Forget_2} \arrow[Holud, bend left=90]{uuu}{ \Forget_1} & & \Comma{ \Id_{\beta} }{ \Forget^{ \Comma{\alpha}{\beta} }_2 } \arrow{dl}[swap]{ \Forget_2 } \arrow[bend right = 90, Holud]{uuu}[swap]{ \Forget_1 } \\
		\calA & \Comma{\alpha}{\beta} \arrow{l}{\Forget_1} \arrow{r}[swap]{\Forget_2} & \calB 
	\end{tikzcd}
\end{equation}
The commutations in \eqref{eqn:ArrowComma3} will be one of the most important theoretical tools in our proofs. The green arrows labeled $U, D, UR, UL, DL, DR$ respectively represent the upper, lower, upper-right, upper-left, lower left and lower-right corners of the object in \eqref{eqn:ijdo4}. The categories $\calA, \calB$ also find their place in this diagram as the smallest ingredients of the  arrow comma category $\ArrowCat{ \Comma{\alpha}{\beta} }$.
We next shift our attention to transformations from between comma categories. 

\paragraph{Comma transformations} Consider a commuting diagram
\begin{equation} \label{eqn:gpd30}
	\begin{tikzcd}
		\calA \arrow{d}{I} \arrow[Holud]{r}{F} & \calB \arrow{d}[swap]{J} & \calC \arrow[Holud]{l}[swap]{G} \arrow{d}{K} \\
		\calA' \arrow[Akashi]{r}[swap]{F'} & \calB' & \calC' \arrow[Akashi]{l}{G'}
	\end{tikzcd}
\end{equation}
in which functors $I,J,K$ connect two comma arrangements $F,G$ and $F'G'$. Then we have

\begin{proposition} [Functors between comma categories] \label{lem:gpd30}
	Consider the arrangement of categories $\calA, \calB, \calC, \calD, \calE$ and functors $F,G,H,I, J$ from \eqref{eqn:gpd30}. Then there is an induced functor between comma categories
	\begin{equation} \label{eqn:odl9}
		\MapA{I,J,K} : \Comma{F}{G} \to \Comma{F'}{G'}, 
	\end{equation}
	where the map between objects and morphisms is as follows :
	\[\begin{tikzcd} [scale cd = 0.7] \paran{a, \phi, c} \arrow{d}{f,g} \\ \paran{a', \phi', c'} \end{tikzcd} \,=\, 
	\begin{tikzcd} [scale cd = 0.7] Fa \arrow{d}[swap]{Ff} \arrow{r}{\phi} & Gc \arrow{d}{Gg} \\ Fa' \arrow{r}[swap]{\phi'} & Gc' \end{tikzcd} \;\mapsto\; 
	\begin{tikzcd} [scale cd = 0.7] 
		F'Ia \arrow{r}{=} & JFa \arrow{d}[swap]{JFf = F'If} \arrow{r}{J\phi} & JGc \arrow{r}{=} \arrow{d}{JGg = G'Kf} & IKc \\ 
		F'Ia' \arrow{r}[swap]{=} & JFa' \arrow{r}[swap]{J\phi'} & JGc' \arrow{r}[swap]{=} & IKc
	\end{tikzcd} \,=\, 
	\begin{tikzcd} [scale cd = 0.7] \paran{Ia, J\phi, Kc} \arrow{d}{If,Kg} \\ \paran{Ia', J\phi', Kc'} \end{tikzcd}\]
	Moreover, the following commutation holds with the marginal functors :
	\begin{equation} \label{eqn:dgbw}
		\begin{tikzcd} [column sep = large]
			\calA \arrow[d, "I"'] & \Comma{F}{G} \arrow[l, "\Forget_1"'] \arrow[r, "\Forget_2"] \arrow[d, "\MapA{I,J,K}"] & \calA \arrow[d, "K"] \\
			\calA' & \Comma{F'}{G'} \arrow[l, "\Forget_1"] \arrow[r, "\Forget_2"'] & \calC'  
		\end{tikzcd}
	\end{equation}
\end{proposition}

The proof of Proposition \ref{lem:gpd30} will be omitted. A particular instance of \eqref{eqn:gpd30} is shown in the center below, 
\begin{equation} \label{eqn:def:MapC}
	\begin{tikzcd} \alpha a \arrow{d}{F} \\ \beta b \end{tikzcd} \imply 
	\begin{tikzcd} [column sep = huge, scale cd = 0.8]
		\calA \arrow{d}{\Id_{\calA}} \arrow{r}{\Id_{\calA}} & \calA \arrow{d}{\Id_{\calA}} & \star \arrow[dashed]{l}[swap]{a} \arrow{d}{F} \\
		\calA \arrow{r}[swap]{\Id_{\calA}} & \calA & \Comma{\alpha}{\beta} \arrow{l}{ \Forget_1^{\Comma{\alpha}{\beta}} }
	\end{tikzcd} \imply 
	\begin{tikzcd}[scale cd = 0.8]
		\Comma{\calA}{a} \arrow[Shobuj, dotted]{dr}{ \MapC{F} := \MapA{ \Id_{\calA}, \Id_{\calA}, F} } \arrow{d}[swap]{\Forget_1} \\
		\calA & \Comma{\Id_{\calA}}{ \Forget_1^{\Comma{\alpha}{\beta}} } \arrow{l}{\Forget_1} 
	\end{tikzcd}
\end{equation}
The leftmost figure in \eqref{eqn:def:MapC} is an object $F$ in $\Comma{\alpha}{\beta}$. The middle diagram presents a simple commutation in which this object is re-interpreted as a functor. Finally, the leftmost figure presents an application of Proposition \ref{lem:gpd30} to this commutation. The dashed arrow in the above diagram indicate that it is are defined via composition. Proposition \ref{lem:gpd30} applied to the commutative diagram in the center leads to the functor $\MapA{ \Id_{\calA}, \Id_{\calA}, F}$ shown on the right. Composition with this functor leads to the functor $\MapC{F}$ shown in green on the right. The top right commutation is a consequence of \eqref{eqn:dgbw}. The action of $\MapC{F}$ can be explained simply as
\[\begin{tikzcd}
	& x \arrow[rd, Holud] \arrow[ld] \arrow[rdd, dashed, Akashi] & \\
	x' \arrow[rr, Holud] \arrow[rrd, dashed, Akashi] & & a \arrow[d, "\phi", Shobuj] \\
	& & a' 
\end{tikzcd} \quad
\begin{tikzcd} {} \arrow[rrr, mapsto, "\Phi_F"] &&& {} \end{tikzcd} \quad 
\begin{tikzcd}
	& \alpha x \arrow[rd, Holud] \arrow[ld] \arrow[rdd, dashed, Akashi] & \\
	\alpha x' \arrow[rr, Holud] \arrow[rrd, dashed, Akashi] & & \alpha a \arrow[d, "\alpha \phi", Shobuj] \arrow[Holud]{r}{F} & \beta b \arrow[Shobuj]{d}{\beta \psi} \\
	& & \alpha a' \arrow[Akashi]{r}{F'} & \beta b'
\end{tikzcd} \]
The yellow and blue sub-diagrams on the actions of $\Phi_{F}$, $\Phi_{F'}$ for different objects $F,F'\in \Comma{\alpha}{\beta}$.
Two other examples of Proposition \ref{lem:gpd30} can be found in the diagram on the left below :
\begin{equation} \label{eqn:0lkd0e}
	\begin{tikzcd}
		\Comma{\alpha}{\beta} \arrow{rr}{ \Forget_2^{ \Comma{\alpha}{\beta}} }  && \calB  & \calB \arrow[l, "="'] \\
		\Comma{\alpha}{\beta} \arrow[u, "="'] \arrow{rr}{ \Forget_2^{ \Comma{\alpha}{\beta}} } \arrow{d}[swap]{ \Forget_2^{ \Comma{\alpha}{\beta}} } && \calB \arrow[u, "="']  \arrow[d, "="'] & \star \arrow[l, "b"] \arrow[u, "b"] \arrow[d, "="'] \\
		\calB \arrow[rr, "="'] && \calB & \star \arrow[l, "b"]
	\end{tikzcd} \imply 
	\begin{tikzcd} \UR{\alpha}{\beta}{\calB} = \Comma{ \Forget_2^{ \Comma{\alpha}{\beta}} }{ \Id_{\calB} } \\ \Comma{ \Forget_2^{ \Comma{\alpha}{\beta}} }{ b } \arrow[u, Shobuj, "\text{Restrict}"'] \arrow[d, Shobuj, "\subseteq"] \\ \Comma{\calB}{b} \end{tikzcd}
\end{equation}
gain Proposition \ref{lem:gpd30} yields the trivial transformations between three categories, as indicated on the right above. We next use this functorial relation $\Phi_F$ between categories to study more complicated arrangements.

\paragraph{The dynamics map} The language of comma categories enable complex arrangements of spaces and transformations to be concisely depicted by the comma notation. Once a comma category is built, one has two forgetful or projection functors from the comma category. One can then use these functors to build comma categories of a higher level of complexity. We have already seen several examples of this constructive procedure over the course of the diagrams \eqref{eqn:ArrowComma2} and \eqref{eqn:ArrowComma3}. Another important instance arises when the functor $\MapC{F}$ from \eqref{eqn:def:MapC} is combined with a part of \eqref{eqn:ArrowComma3} to give
\[\begin{tikzcd}
	\Comma{\calA}{a} \arrow{r}[swap]{\MapC{F}} & \Comma{\alpha}{\beta} & & \ArrowCat{ \Comma{\alpha}{\beta} } \arrow{ll}{ D_{\alpha, \beta} }
\end{tikzcd}\]
Now consider any element $a' \xrightarrow{f} a$ from the slice category $\Comma{\calA}{a}$, which we just represent as $f$. This can be represented as a functor from the one point category $\star$. This leads to 
\[\begin{tikzcd}
	& \star \arrow[dashed]{d} \arrow[blue]{dl}[swap]{f} & & \\
	\Comma{\calA}{a} \arrow{r}[swap]{\MapC{F}} & \Comma{\alpha}{\beta} & & \ArrowCat{ \Comma{\alpha}{\beta} } \arrow{ll}{ D_{\alpha, \beta} }
\end{tikzcd}\]
This arrangement is also a diagram in $\CatCat$, the category of small categories. As a result we can construct its pull back, which is shown below in green :
\[\begin{tikzcd}
	& \star \arrow[dashed]{d} \arrow[blue]{dl}[swap]{f} & & \Shobuj{ \calZ \paran{ f, F } } \arrow[Shobuj, ll] \arrow[d, Shobuj] \\
	\Comma{\calA}{a} \arrow{r}[swap]{\MapC{F}} & \Comma{\alpha}{\beta} & & \ArrowCat{ \Comma{\alpha}{\beta} } \arrow{ll}{ D_{\alpha, \beta} }
\end{tikzcd}\]
The category $\calZ \paran{ f, F }$ is the full subcategory of $\ArrowCat{ \Comma{\alpha}{\beta} }$ whose objects are pairs $(F',g)$ such that
\[\begin{tikzcd}
	\alpha a' \arrow[d, "\beta f"'] \arrow[r, "F'"] & \beta b' \arrow[d, "\beta g"] \\
	\alpha a \arrow[r, "F"'] & \beta b 
\end{tikzcd}\]
The upper horizontal arrow can be recovered via the functor $U_{\alpha, \beta}$. This functor can be added to our previous arrangement to get :
\begin{equation} \label{eqn:def:tau:2}
	\begin{tikzcd}
		& \star \arrow[dashed]{d} \arrow[blue]{dl}[swap]{f} & & \Shobuj{ \calZ \paran{ f, F } } \arrow[Shobuj, ll] \arrow[d, Shobuj] \arrow[dashed, blue]{drr}{ \zeta_{f,F} } \\
		\Comma{\calA}{a} \arrow{r}[swap]{\MapC{F}} & \Comma{\alpha}{\beta} & & \ArrowCat{ \Comma{\alpha}{\beta} } \arrow{ll}{ D_{\alpha, \beta} } \arrow{rr}[swap]{ U_{\alpha, \beta} } & & \Comma{\alpha}{\beta}
	\end{tikzcd}
\end{equation}
The functor $\zeta_{f,F}$ which is created via composition, directly yields the functor we are looking for :
\begin{equation} \label{eqn:spd9f}
	\tau_F(f) := \lim \zeta_{f,F} .
\end{equation}
Equation \eqref{eqn:spd9f} is an alternative and easy route to construct the functor $\tau_F$ from Theorem \ref{thm:1}. However, this construction does not reveal the functorial nature of the correspondence between $f$ and $\lim \zeta_{f,F}$. In the following section, we describe a different route to establishing functoriality. 
\section{Adjointness in commas} \label{sec:comma_adj}

In this section the categorical properties of various comma categories and forgetful functors will be examined. The first is a classic result from Category theory :

\begin{lemma} \label{lem:comma_cocmplt} 
	\cite[Thm 3]{RydeheardBurstall988} Let $\alpha : \calA \to  \calC$ and $\beta : \calB \to  \calC$ be functors with $\alpha$ (finitely) continuous. If $\calA$ and $\calB$ are (finitely) complete, then so is the comma category $\Comma{\alpha}{\beta}$
\end{lemma}

One immediate consequence of Lemma \ref{lem:comma_cocmplt} is :

\begin{lemma} \label{lem:alph_bet_cmplt}
	The comma category $\Comma{\alpha}{\beta}$ is complete.  
\end{lemma}

We also have :

\begin{lemma} \label{lem:UR_cmplt}
	The upper right category $\UR{\alpha}{\beta}{\calB}$ is complete.  
\end{lemma}

Lemma \ref{lem:UR_cmplt} follows directly from the construction of $\UR{\alpha}{\beta}{\calB}$ as a comma category in Section \ref{sec:comma_arr}, the continuity of the identity functor, and Lemma \ref{lem:comma_cocmplt}. The following basic lemma is an useful tool in establishing the existence of right adjoints. 

\begin{lemma} [Right inverse as right adjoint] \label{lem:RIRA}
	Suppose $F:\calP \to \calQ$ and $G:\calQ \to \calP$ are two functors such that $FG = \Id_{\calQ}$ and $\Id_{\calP} \Rightarrow GF$. Then $F,G$ are left and right adjoints of each other.
\end{lemma}

\blue{These insights lead to two technical results. The first one is :}

\begin{lemma} \label{lem:frgt_adjnt}
	Consider an arrangement of categories and functors $\calP \xrightarrow{P} \calR \xleftarrow{Q} \calQ$. If $\calP$ and $\calR$ have initial elements $0_{\calP}$ and $1_{\calR}$ respectively, and $P \paran{0_{\calP}} = 0_{\calR}$, then the functor $\Forget_2 : \Comma{P}{Q}$ has a left adjoint given by
	\[ \paran{ \Forget_2 }^{(L)} : \calQ \to \Comma{P}{Q}, \quad q \mapsto \; \begin{tikzcd} P 0_{\calP} = 0_{\calR} \arrow[d, "!_{Qq}"] \\ Qq \end{tikzcd} \]
	In fact, $\paran{ \Forget_1 }^{(L)}$ is a right inverse of $\Forget_1$, i.e., $\Forget_2 \circ \paran{ \Forget_2 }^{(L)} = \Id_{\calQ}$. 
\end{lemma}

\begin{proof} 
	Since $\paran{ \Forget_1 }^{(R)}$ is a right inverse of $\Forget_1$ by Lemma \ref{lem:RIRA}, it only remains to be shown that there is a natural transformation $\paran{ \Forget_2 }^{(L)} \circ \Forget_2 \Rightarrow \Id_{\Comma{P}{Q}}$, called the \emph{counit}. The diagram on the left below traces the action of this composite functor on an object $F$ (blue) of $\Comma{P}{Q}$ into an object (green) of $\Comma{P}{Q}$ :
	\[\begin{tikzcd} Pp \arrow[d, "F"', Akashi] \\ Qq \end{tikzcd}
	\;\begin{tikzcd} {} \arrow[rr, mapsto, "\Forget_2"] && {} \end{tikzcd} \; 
	q \;\begin{tikzcd} {} \arrow[rr, mapsto, "\paran{ \Forget_2 }^{(L)}"] && {} \end{tikzcd} \; 
	\begin{tikzcd} P 0_{\calP} = 0_{\calR} \arrow[d, "!_{Qq}", Shobuj] \\ Qq \end{tikzcd} ; \quad 
	\begin{tikzcd}
		P0_{\calP} \arrow[d, "P(!_p)"', Holud] \arrow[r, "!_{Qq}", Shobuj] & Qq \arrow[d, "Q \Id_q", Holud]\\
		Pp \arrow[r, "F", Akashi] & Qq
	\end{tikzcd}\]
	The diagram on the left demonstrates a commutation arising out of the initial element preserving property. The descending yellow morphisms together constitute the connecting morphism of the counit transformation we seek. This completes the proof.
\end{proof}

The second technical results is :

\begin{lemma} \label{lem:jd9lk3}
	In the arrangement of \eqref{eqn:abc}, the forgetful functor 
	\[ \Forget_{1} : \UR{\alpha}{\beta}{\calB} := \Comma{ \Forget_2^{ \Comma{\alpha}{\beta} } }{ \Id_{\calB} } \to \Comma{\alpha}{\beta} , \]
	has a left adjoint 
	\[ \Forget_{1}^{(L)} :  \Comma{\alpha}{\beta} \to \UR{\alpha}{\beta}{\calB} , \quad 
	\begin{tikzcd} \alpha a \arrow[d, "F"'] \\ \beta b \end{tikzcd} \; 
	\begin{tikzcd} {}  \arrow[rr, mapsto] && {} \end{tikzcd} \; 
	\begin{tikzcd} \alpha a \arrow[r, "F"'] & \beta b \arrow[d, "\beta \Id_b"]\\ & \beta b  \end{tikzcd} .
	\]
\end{lemma}

\begin{proof} The functor $\Forget_{1}^{(L)}$ is clearly a right inverse of $\Forget_{1}$. Thus again by Lemma \ref{lem:RIRA} , it is enough to show the existence of a natural transformation as shown on the left below :
	\[ \eta : \Forget_1^{(L)} \circ \Forget_1 \Rightarrow \Id , \quad
	\begin{tikzcd} \alpha a \arrow[d, "F"', Akashi] \\ \beta b \end{tikzcd}, \, 
	\begin{tikzcd} b  \arrow[d, "\phi"', Akashi] \\ b' \end{tikzcd} , \; 
	\begin{tikzcd} {}  \arrow[rr, "\eta", mapsto] && {} \end{tikzcd} \;
	\begin{tikzcd}
		& \alpha a \arrow[r, "F", Holud] \arrow[ld, "="] & \beta b \arrow[ld, "="'] \arrow[d, "\beta \Id_b", Holud] & \\
		\alpha a \arrow[r, "F"', Akashi] & \beta b \arrow[d, "\phi"', Akashi] & \beta b \arrow[ld, "\phi"] & \\
		& \beta b' & & 
	\end{tikzcd}\]
	The connecting morphisms of this natural transformation is shown above on the right. 
\end{proof}

\blue{ The reader is once again referred to the outline presented in Figure \ref{fig:outline1}. The main results that remain to be proven are Theorems \ref{thm:1} and \ref{thm:2}. They are proved by a final, complicated diagram presented later in \eqref{eqn:nc0s3}. Lemmas \ref{lem:frgt_adjnt}, \ref{lem:jd9lk3} help in the construction of this diagram. }

We end this section with one final observation about the upper-right corner category. Recall the inclusion functor from \eqref{eqn:0lkd0e}. It leads to a pull-back square :
\[\begin{tikzcd}
	\Comma{\Forget_2^{ \Comma{\alpha}{\beta}}}{b}  \arrow[r, "\subset"] \arrow[d, ""] & \UR{\alpha}{\beta}{\calB} \arrow[d, "\Forget_2"] \\
	\star \arrow[r, "b"'] & \calB
\end{tikzcd}\]
Now consider any functor $\calT: \calP \to \UR{\alpha}{\beta}{\calB}$ such that $\Forget_2 \circ \calT \equiv b$, for some object $b$ of $\calB$. Then the commutation on the left below is satisfied :
\begin{equation} \label{eqn:T_T2_T3}
	\begin{tikzcd}
		\calP \arrow[r, "\calT"] \arrow[d, ""] & \UR{\alpha}{\beta}{\calB} \arrow[d, "\Forget_2"] \\
		\star \arrow[r, "b"'] & \calB
	\end{tikzcd} \imply 
	\begin{tikzcd}
		\calP \arrow[drr, "\calT", bend left=20] \arrow[ddr, "", bend right=20] \arrow[Akashi, dr, "\calT'"] \arrow[Akashi, dd, "\calT''"', dashed]  \\
		& \Comma{\Forget_2^{ \Comma{\alpha}{\beta}}}{b} \arrow[dl, Holud, "\text{Restrict}"] \arrow[r, "\subset"] \arrow[d, ""] & \UR{\alpha}{\beta}{\calB} \arrow[d, "\Forget_2"] \\
		\Comma{\calB}{b} & \star \arrow[r, "b"'] & \calB
	\end{tikzcd}
\end{equation}
This commutation must factor through the pullback square as shown by the blue arrow in the diagram in the middle. Thus given any functor $\calT$ as above,  \eqref{eqn:T_T2_T3} says that $\calT$ factors through a map $\calT'$ mapping into $\Comma{\Forget_2^{ \Comma{\alpha}{\beta}}}{b}$, a subcategory of $\UR{\alpha}{\beta}{\calB}$. Since the category $\Comma{\Forget_2^{ \Comma{\alpha}{\beta}}}{b}$ can be further restricted to the slice $\Comma{\calB}{b}$, $\calT'$ extends to a functor mapping into $\Comma{\calB}{b}$.

\section{The Lim-Pre construction} \label{sec:LimPre}

At this stage we can start constructing the functors declared in Theorem \ref{thm:1} and \eqref{eqn:thm:2}. These will be constructed by taking various limits. For that purpose we need to established the completeness and continuity of various categories and functors involved. We start with a classic result from Category theory :

\begin{lemma} [Right adjoints preserve limits] \label{lem:RAPL}
	\cite[Thm 4.5.3]{Riehl_context_2017} If a functor $F:\calP\to \calQ$ has a left adjoint, then for any diagram $\Psi : J\to \calP$, if $\lim \psi$ exists, then $\lim (F\circ \Psi) = F \paran{ \lim \Psi }$. 
\end{lemma}

Consider an arrangement $\calY \xleftarrow{F} \calX \xrightarrow{G} \calZ$. Given an object $y$ of $\calY$, one can create the composite functor :
\begin{equation} \label{eqn:def:preFG}
	\begin{tikzcd}
		\Comma{F}{y} \arrow[dashed]{rr}{ \green{ \Pre{F}{G}(y) } } \arrow{d}[swap]{ \Forget_{ \Comma{F}{y} } } && Z \\
		X \times \{\star\} \arrow{rr}{\cong} && X \arrow{u}[swap]{ G }
	\end{tikzcd}
\end{equation}
Recall that $\Comma{ \CatCat }{\calZ}$ is the left-slice of $\calZ$ in the category $\CatCat$ of small categories. Its objects are thus all possible functors with codomain $\calZ$. The construction \eqref{eqn:def:preFG} thus gives us a functor
\begin{equation} \label{eqn:def:Pre}
	\Pre{F}{G} : \calY \to \Comma{ \CatCat }{\calZ}, \; y\mapsto \paran{ \Comma{F}{y}, \Pre{F}{G}(y) } ;\;  
	\begin{tikzcd} y \arrow{d}{\psi} \\ y' \end{tikzcd} \,\mapsto\, 
	\begin{tikzcd}
		\Comma{F}{y} \arrow[bend left=20]{drr}{ \Pre{F}{G}(y) } \arrow[bend right=10]{dr}{ \Forget_{ \Comma{F}{y} } } \arrow{dd}{\psi\circ}\\
		& X \arrow{r}{G} & Z\\
		\Comma{F}{y'} \arrow[bend left=10]{ur}[swap]{ \Forget_{ \Comma{F}{y'} } } \arrow[bend right=20]{urr}[swap]{ \Pre{F}{G}(y') }
	\end{tikzcd}
\end{equation}
Now suppose that $\calZ$ is a complete category. The collection of all diagrams in $\calZ$, which are functors $F:J\to \calZ$ is the left slice of $\calZ$ within $\CatCat$ the category of small categories. We then have the following result from basic category theory :

\begin{lemma} \label{lem:id30l}
	Given a complete category $\calZ$, there is a functor $\lim : \Comma{\CatCat}{\calZ} \to \calZ$ which maps each diagram $F:J\to \calZ$ into $\lim F$.
\end{lemma}

For a complete category $\calZ$, the $\lim$ functor can be used to extend the functor from \eqref{eqn:def:Pre} into the dashed green arrow as shown below : 
\begin{equation} \label{eqn:LimPre}
	\begin{tikzcd}
		\calY \arrow[drr, Shobuj, "\LimPre_{F,G}"', dashed] \arrow[rr, "\Pre{F}{G}"] && \Comma{ \CatCat }{Z} \arrow[d, "\lim"] \\
		&& \calZ
	\end{tikzcd}
\end{equation}
This construction $\LimPre_{F,G}$ is one of the main innovations in this paper. The use of colimits instead of limits would have yielded the right Kan extension of $F$ along $G$. for any object $y$ of $\calY$, $\LimPre_{F,G}(y)$ is the limit point of the functor $G$ restricted to the left slice $\Comma{F}{y}$. 

Given two complete categories $\calZ, \calZ'$ a functor $\Psi : \calZ \to \calZ'$ is called \emph{limit preserving} if the following commutation holds : 
\begin{equation} \label{eqn:svo02l}
	\begin{tikzcd}
		\Comma{\CatCat}{\calZ} \arrow[d, "\lim"'] \arrow[r, "F\circ"] & \Comma{\CatCat}{\calZ'} \arrow[d, "\lim"] \\
		\calQ \arrow[r, "F"] & \calZ'
	\end{tikzcd}
\end{equation}
\blue{The $\LimPre$ construction involves two functors $R,Q$ with the same domain category. The second functor $Q$ may be extended to a different codomain by composition with some functor $F$. The next lemma presents a simple condition under which the $\LimPre$ applied to a composition of $Q$ with $F$ coincides with the composition of $\LimPre_{R,Q}$ with $F$.}

\begin{lemma} \label{lem:LimPre:extend}
	Given an arrangement of functors 
	\[\begin{tikzcd}
		\calR & \calP \arrow[r, "Q"] \arrow[l, "R"'] & \calQ \arrow[r, "F"] & \calQ'
	\end{tikzcd}\]
	in which $F$ is limit preserving, the following commutation holds between the $\LimPre$ functors :
	\[\begin{tikzcd}
		\calP \arrow[rrr, bend left=30, "F\circ Q"] \arrow[rr, "Q"'] \arrow[d, "R"'] && \calQ \arrow[r, "F"'] & \calQ' \\
		\calR \arrow[rrr, bend right=30, "\LimPre_{R,F\circ Q}"'] \arrow[rr, "\LimPre_{R,Q}"] && \calQ \arrow[r, "F"] & \calQ'
	\end{tikzcd}\]
\end{lemma}

\begin{proof} Note that for every $y\in ob(\calY)$, 
	\[\begin{split}
		\LimPre_{R,F\circ Q} (y) &= \lim \Pre{R}{F\circ Q}(y) = \lim \Pre{R}{F\circ \Pre{R}Q}(y) \\
		& = F\circ \lim \Pre{R}{\Pre{R}Q}(y) = F\circ \LimPre_{R,Q} (y) ,
	\end{split}\]
	where the second last inequality holds from the limit preservation property.
\end{proof}

\blue{ The reader is once again referred to the outline presented in Figure \ref{fig:outline1}. The statement of Lemma \ref{lem:LimPre:extend} is essentially a commutation. This commutation will be seen to occur multiple times in the diagram presented later in \eqref{eqn:nc0s3}. This diagram is the final step towards proving Theorems \ref{thm:1} and \ref{thm:2}. At this point we are ready to begin the proof. }

\section{Proof of Theorems \ref{thm:1} and \ref{thm:2}} \label{sec:proof:2}

The proofs of the two theorems shall be derived simultaneously.
We begin the proof by drawing a portion of \eqref{eqn:ArrowComma3}. 
\[\begin{tikzcd} []
	\ArrowCat{ \Comma{\alpha}{\beta} } \arrow[dd, "DL"'] \arrow[rr, "UR", Shobuj] \arrow[bend left = 30, Holud, "\Forget_1", rrrr] && \UR{\alpha}{\beta}{\calB} \arrow[rr, "\Forget_1"] && \Comma{\alpha}{\beta}  \\
	&& && &&  \\
	\DL{\alpha}{\beta}{\calA} 
\end{tikzcd}\]
Since the categories $\Comma{\alpha}{\beta}$ and $\UR{\alpha}{\beta}{\calB}$ are complete by Lemmas \ref{lem:comma_cocmplt} and \ref{lem:UR_cmplt} respectively, we can create the $\LimPre$ constructions of these functors, as shown below : 
\[\begin{tikzcd} []
	\ArrowCat{ \Comma{\alpha}{\beta} } \arrow[dd, "DL"'] \arrow[rr, "UR", Shobuj] \arrow[bend left = 30, Holud, "\Forget_1", rrrr] && \UR{\alpha}{\beta}{\calB} \arrow[rr, "\Forget_1"] && \Comma{\alpha}{\beta}  \\
	&& && &&  \\
	\DL{\alpha}{\beta}{\calA} \arrow[rrrr, bend right=20, Holud, dotted, "\gamma_F"'] \arrow[rr, "\bar{\gamma}_F"', Shobuj, dotted] && \UR{\alpha}{\beta}{\calB} \arrow[rr, "\Forget_1"] && \Comma{\alpha}{\beta} 
\end{tikzcd}\]
Each colored dotted arrow is the $\LimPre$ construction corresponding to the functor of the same color on the top row. The commutation between the $\LimPre$ constructions in the second row holds by Lemma \ref{lem:frgt_adjnt}, Lemma \ref{lem:jd9lk3} and Lemma \ref{lem:LimPre:extend}. We now add some of the peripheral commutations of \eqref{eqn:ArrowComma3} to get
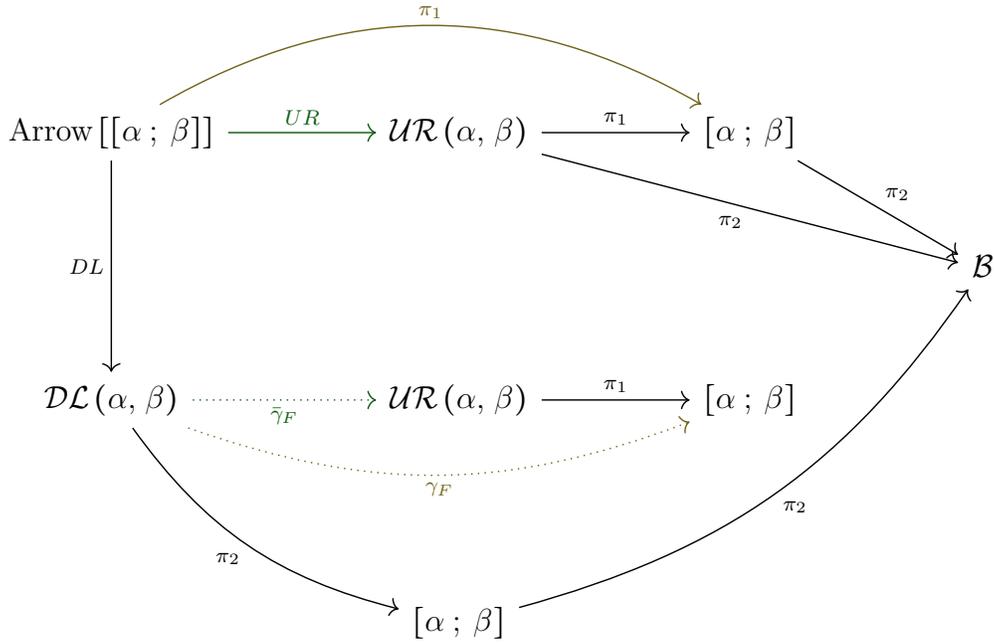
\[\begin{tikzcd} [row sep = large, scale cd = 0.6]
	\ArrowCat{ \Comma{\alpha}{\beta} } \arrow[dd, "DL"'] \arrow[rr, "UR", Shobuj] \arrow[bend left = 30, Holud, "\Forget_1", rrrr] && \UR{\alpha}{\beta}{\calB} \arrow[drrrr, "\Forget_2"'] \arrow[rr, "\Forget_1"] && \Comma{\alpha}{\beta} \arrow[drr, "\Forget_2"] \\
	&& && && \calB \\
	\DL{\alpha}{\beta}{\calA} \arrow[ddrr, "\Forget_2"', bend right=20] \arrow[rrrr, bend right=20, Holud, dotted, "\gamma_F"'] \arrow[rr, "\bar{\gamma}_F"', Shobuj, dotted] && \UR{\alpha}{\beta}{\calB} \arrow[rr, "\Forget_1"] && \Comma{\alpha}{\beta} \\ \\
	&& \Comma{\alpha}{\beta} \arrow[uuurrrr, "\Forget_2"', bend right=20] \\
\end{tikzcd}\]
Again, by Lemma \ref{lem:frgt_adjnt} and Lemma \ref{lem:LimPre:extend} we can fill in :
\[\begin{tikzcd} [row sep = large, scale cd = 0.6]
	\ArrowCat{ \Comma{\alpha}{\beta} } \arrow[dd, "DL"'] \arrow[rr, "UR", Shobuj] \arrow[bend left = 30, Holud, "\Forget_1", rrrr] && \UR{\alpha}{\beta}{\calB} \arrow[drrrr, "\Forget_2"'] \arrow[rr, "\Forget_1"] && \Comma{\alpha}{\beta} \arrow[drr, "\Forget_2"] \\
	&& && && \calB \\
	\DL{\alpha}{\beta}{\calA} \arrow[rrrr, bend right=20, Holud, dotted, "\gamma_F"'] \arrow[ddrr, "\Forget_2"', bend right=20] \arrow[rr, "\bar{\gamma}_F"', Shobuj, dotted] && \UR{\alpha}{\beta}{\calB} \arrow[rr, "\Forget_1"] \arrow[urrrr, "\Forget_2"] && \Comma{\alpha}{\beta} \arrow[urr, "\Forget_2"'] \\ \\
	&& \Comma{\alpha}{\beta} \arrow[uuurrrr, "\Forget_2"', bend right=20] \\
\end{tikzcd}\]
We now add the functor $\MapC{F}$ from \eqref{eqn:def:MapC} to this diagram : 
\[\begin{tikzcd} [row sep = large]
	\ArrowCat{ \Comma{\alpha}{\beta} } \arrow[dd, "DL"'] \arrow[rr, "UR", Shobuj] \arrow[bend left = 30, Holud, "\Forget_1", rrrr] && \UR{\alpha}{\beta}{\calB} \arrow[drrrr, "\Forget_2"'] \arrow[rr, "\Forget_1"] && \Comma{\alpha}{\beta} \arrow[drr, "\Forget_2"] \\
	&& && && \calB \\
	\DL{\alpha}{\beta}{\calA} \arrow[rrrr, bend right=20, Holud, dotted, "\gamma_F"'] \arrow[ddrr, "\Forget_2"', bend right=20] \arrow[rr, "\bar{\gamma}_F"', Shobuj, dotted] && \UR{\alpha}{\beta}{\calB} \arrow[rr, "\Forget_1"] \arrow[urrrr, "\Forget_2"] && \Comma{\alpha}{\beta} \arrow[urr, "\Forget_2"'] \\ \\
	&& \Comma{\alpha}{\beta} \arrow[uuurrrr, "\Forget_2"', bend right=20] \\
	\Comma{\calA}{a}  \arrow[Itranga]{uuu}{ \MapC{F} } \arrow[rr] && \star \arrow[u, Itranga, "F"] \arrow[uuuurrrr, Itranga, "b"', bend right=30] && 
\end{tikzcd}\]
Now set $\bar{\tau_F} = \bar{\gamma}_F \circ \MapC{F}$ and $\tau_F = \gamma_F \circ \MapC{F}$. Note that this creates a commutation :
\begin{equation} \label{eqn:BarTau_b}
	\begin{tikzcd}
		\Comma{\calA}{a} \arrow[rr, "\bar{\tau}_F"] \arrow[d, ""] && \UR{\alpha}{\beta}{\calB} \arrow[d, "\Forget_2"] \\
		\star \arrow[rr, "b"'] && \calB
	\end{tikzcd}
\end{equation}
This is precisely the commutation described on the left of \eqref{eqn:T_T2_T3}. Thus 
the conclusions of \eqref{eqn:T_T2_T3}  along with \eqref{eqn:0lkd0e} hold and we get :
\begin{equation} \label{eqn:nc0s3}
	\begin{tikzcd} [row sep = large, scale cd = 0.8]
		\ArrowCat{ \Comma{\alpha}{\beta} } \arrow[dd, "DL"'] \arrow[rr, "UR", Shobuj] \arrow[bend left = 30, Holud, "\Forget_1", rrrr] && \UR{\alpha}{\beta}{\calB} \arrow[drrrr, "\Forget_2"'] \arrow[rr, "\Forget_1"] && \Comma{\alpha}{\beta} \arrow[drr, "\Forget_2"] \\
		&& && && \calB \\
		\DL{\alpha}{\beta}{\calA} \arrow[rrrr, bend right=20, Holud, dotted, "\gamma_F"'] \arrow[ddrr, "\Forget_2"', bend right=20] \arrow[rr, "\bar{\gamma}_F"', Shobuj, dotted] && \UR{\alpha}{\beta}{\calB} \arrow[rr, "\Forget_1"] \arrow[urrrr, "\Forget_2"] && \Comma{\alpha}{\beta} \arrow[urr, "\Forget_2"'] \\ \\
		&& \Comma{\alpha}{\beta} \arrow[uuurrrr, "\Forget_2"', bend right=20] & \Comma{\Forget_2^{ \Comma{\alpha}{\beta}}}{b} \arrow[dr, Akashi, "\text{Restrict}"] \arrow[luu, Akashi, "\subseteq"] \\
		\Comma{\calA}{a} \arrow[Akashi, rrru, bend right=40, "\bar{\tau}_F"'] \arrow[rrrr, bend right=30, Shobuj, "\Dyn_F"'] \arrow[Itranga]{uuu}{ \MapC{F} } \arrow[rr] && \star \arrow[u, Itranga, "F"] \arrow[uuuurrrr, Itranga, "b"', bend right=30] && \Comma{\calB}{b} \arrow[uuuurr, "\Forget_2"', bend right=30]
	\end{tikzcd}
\end{equation}
The commutations in \eqref{eqn:thm:2} are included within the commutation of \eqref{eqn:nc0s3}. The claim of minimality in Theorem \ref{thm:1} follows from the construction of $\gamma_F$ as a limit. The commutation diagram in \eqref{eqn:BarTau_b} links this minimal comma object to minimal commutation squares completing $\begin{tikzcd} \alpha a' \arrow[d, "\alpha f"'] \\ \alpha a \arrow[r, "F"'] & \beta b \end{tikzcd}$. 
This completes the proof of Theorems \ref{thm:1} and \ref{thm:2}. \qed 

\bibliographystyle{\Path unsrt}
\bibliography{\Path References,Ref}
\end{document}